\documentclass[12pt]{amsart}
\usepackage{amsmath}
\usepackage{hyperref}
\usepackage{amssymb}
\usepackage{latexsym}
\usepackage{amscd}
\usepackage{graphicx} 
\usepackage{amsthm}
\usepackage{mathrsfs}
\usepackage{xypic}
\usepackage{bm}

\newdimen\AAdi%
\newbox\AAbo%
\def\AAk#1#2{\s_etbox\AAbo=\hbox{#2}\AAdi=\wd\AAbo\kern#1\AAdi{}}%
\def\AAr#1#2#3{\s_etbox\AAbo=\hbox{#2}\AAdi=\ht\AAbo\raise#1\AAdi\hbox{#3}}%
\font\tenmsb=msbm10 at 12pt \font\sevenmsb=msbm7 at 8pt
\font\fivemsb=msbm5 at 6pt
\newfam\msbfam
\textfont\msbfam=\tenmsb \scriptfont\msbfam=\sevenmsb
\scriptscriptfont\msbfam=\fivemsb
\def\Bbb#1{{\tenmsb\fam\msbfam#1}}
\newtheorem{thm}{Theorem}[section]
\newtheorem{lem}{Lemma}[section]
\newtheorem{cor}{Corollary}[section]

\newtheorem{pro}{Proposition}[section]
\newtheorem{defi}{Definition}[section]

\newcommand{\ba}{\begin{array}}
\newcommand{\ea}{\end{array}}

\newcommand{\Section}[2]{\setcounter{equation}{0}
\allowdisplaybreaks
\section[#1]{#2}}

\def\n{\nabla}

\def\ir#1{\mathbb R^{#1}}

\def\f#1#2{\frac{#1}{#2}}

\def\grs#1#2{\bold G_{#1,#2}}

\def\mc#1{\mathcal{#1}}

\def\td{\tilde}

\def\a{\alpha}
\def\be{\beta}

\def\p#1{\partial #1}

\def\de{\delta}
\def\De{\Delta}

\def\ep{\varepsilon}

\def\g{\gamma}
\def\k{\kappa}
\def\la{\lambda}

\def\om{\omega}
\def\Om{\Omega}
\def\th{\theta}
\def\Th{\Theta}
\def\si{\sigma}
\def\Si{\Sigma}

\def\Hess{\mbox{Hess}}
\def\R{\Bbb{R}}

\def\lan{\langle}
\def\ran{\rangle}
\def\ra{\rightarrow}

\def\aint#1{-\hskip -4.5mm\int_{#1}}
\def\V{\mbox{Vol}}
\def\ol{\overline}

\begin{document}
\title
[Harmonic maps into spheres] {The regularity of harmonic maps into
  spheres and applications to Bernstein problems}

\author
[J. Jost, Y. L. Xin and Ling Yang]{J. Jost, Y. L. Xin and Ling
Yang}
\address{Max Planck Institute for Mathematics in the
Sciences, Inselstr. 22, 04103 Leipzig, Germany.}
\email{jost@mis.mpg.de}
\address {Institute of Mathematics, Fudan University,
Shanghai 200433, China.} \email{ylxin@fudan.edu.cn}
\address{Max Planck Institute for Mathematics in the
Sciences, Inselstr. 22, 04103 Leipzig, Germany.}
\email{lingyang@mis.mpg.de}
\thanks{The second named author is grateful to the Max Planck
Institute for Mathematics in the Sciences in Leipzig for its hospitality and  continuous support.
He is also partially supported by NSFC
and SFMEC }
\begin{abstract}
We show the regularity of, and derive a-priori estimates for
(weakly) harmonic maps from a Riemannian manifold into a
Euclidean sphere under the assumption that the image avoids some
neighborhood of a half-equator. The proofs combine constructions
of strictly convex functions and the regularity theory of
quasi-linear elliptic systems.

We apply these results to the spherical and Euclidean Bernstein
problems for minimal hypersurfaces, obtaining new conditions
under which compact minimal hypersurfaces in spheres or complete
minimal hypersurfaces in Euclidean spaces are trivial.

\end{abstract}

\renewcommand{\subjclassname}{%
  \textup{2000} Mathematics Subject Classification}
\date{}
\maketitle
\Section{Introduction}{Introduction}

Harmonic maps into spheres need not be regular. The basic example is
due to \cite{h-k-w}: The map
\begin{equation}
\label{0}
\frac{x}{|x|}:\mathbb{R}^{n} \to S^{n-1}
\end{equation}
has a singularity at the origin 0, while having finite energy on
finite balls for
$n\ge 3$ (it thus is a so-called weakly harmonic map). This example can be modified by embedding $S^{n-1}$ as an
equator into $S^n$, and the composed map then is a singular harmonic
map from $\R^n$ to the sphere $S^n$ with image contained in an
equator. This equator is the boundary of a closed hemisphere. In
contrast to this phenomenon, Hildebrandt-Kaul-Widman \cite{h-k-w}
proved the regularity of weakly harmonic maps whose image is contained in
some compact subset of an open hemisphere. Hidlebrandt-Jost-Widman
\cite{h-j-w} then derived a-priori estimates for harmonic maps in that
situation. The example then shows that
these results are optimal in the sense that the open hemisphere cannot be
replaced by a closed one. It was then the general opinion that for
general harmonic maps (not necessarily energy minimizing, in which
case the method of \cite{u-s2} yields additional results, see for
instance \cite{so2}), this is the
best that one can do.

Here, we show that one can do substantially better. In fact, we shall
show that weakly harmonic maps into a sphere are regular, and satisfy
a-priori estimates, under the condition that their image be contained
in a compact subset of the complement of half of an equator, that is,
in the complement of half of a totally geodesic $(n-1)$-dimensional
subsphere. Of course, our condition still rules out the counterexample
of \cite{h-k-w}.

Our condition is presumably optimal, for the
following reason. The basic principle underlying the regularity theory
for harmonic maps is the fact that the composition of a harmonic map
with a convex function on the target yields a subharmonic
function, and in the case of weakly harmonic maps, we obtain a weakly
subharmonic function. One then exploits the maximum principle for such a
(weakly) subharmonic function, or in more refined
schemes, Moser's Harnack inequality to derive estimates for the
original (weakly) harmonic map. This obviously depends on the careful
utilization of the geometric properties of the convex function. In
fact, for the full regularity scheme, it is not sufficient to have a
single convex function, but we rather need a family of such convex
functions. More precisely, for each point in the target, we need a
strictly convex function that assumes its minimum at that particular
point. (In the original scheme of \cite{h-k-w,h-j-w}, the authors
worked with squared distance functions from points in the
target. Therefore, in an open hemisphere, one could only have such
functions that were strictly convex only on some part of that
hemisphere, depending on where their minimum was located. This
necessitated an iteration scheme whose idea was to show that the image
of a (weakly) harmonic map gets smaller in a controlled manner when
one decreases its domain. Some simplification can be achieved by the
construction of Kendall \cite{k1} of  strictly convex functions
on arbitrary compact subsets of an open hemisphere with a minimum at
some prescribed point.)

Thus, an essential part of the scheme developed in the present paper
consists in the construction of such strictly convex functions on
arbitrary compact subsets of the complement of a half-equator. This
is rather subtle. In fact, the functions we construct will depend on
the compact set $K$ in question, and none of them will be convex on
the entire open complement $\Bbb{V}$ of the half-equator. (In fact, there
is no $D\subset S^n$ that is a maximal domain of definition of a
strictly convex function, see \cite{a-b}.) We just fine-tune them in such a way that intersections of
their level sets with $K$ are convex, while these level sets are
allowed to be concave on $\Bbb{V}\setminus K$, so as to turn around the
boundary of the half-equator.

The reason why our result is presumably optimal then is that as soon
we enlarge the open set $\Bbb{V}$, it will contain a closed geodesic. Since
strictly convex functions are strictly monotonic along geodesic arcs,
a set containing a closed geodesic cannot carry a strictly convex
function. Therefore, our construction will no longer work then. Since
the presence of strictly convex functions is essentially necessary for
harmonic map regularity, this then seems to prohibit any general
regularity result, let alone an explicit estimate.

Still, even with those strictly convex functions, the regularity
theory is difficult and subtle, and we need to utilize  the most
advanced tools available in the literature. In particular, we use the
Green test function technique and image shrinking method employed in \cite{h-k-w,h-j-w} and the
generalization of that scheme in \cite{g-j}, the estimates for Green
functions of \cite{g-w,b-m} that depend on Moser's Harnack inequality
\cite{M}, the telescoping trick of \cite{g-g,g-h}, and the Harnack
inequality method of \cite{J} that converts convexity assumptions on
the target into energy and oscillation controls for harmonic maps. A crucial point is that
our estimates will not depend on the energy of the harmonic map to be
estimated. Therefore, in particular, we do not need to make any energy
minimizing assumption, and when we turn to global issues, we only need
to assume the map to have locally finite energy, but not necessarily globally.

Following the scheme of \cite{h-j-w}, we can therefore apply our a-priori
estimates to the Bernstein problem for minimal hypersurfaces in
spheres and Euclidean spaces. The connection between such Bernstein
problems and harmonic maps into spheres comes from the Ruh-Vilms
theorem \cite{r-v} that says that the Gauss map of a minimal
hypersurface is a {\it harmonic} map (with values in a
sphere). Showing that the original minimal hypersurface is trivial (a
totally geodesic subsphere or a hyperplane, resp.) then is reduced to
showing that the Gauss map is constant. In order to apply our results,
we therefore have to show that the Gauss map is constant under the
assumption that its image is contained in a compact subset of the
complement of a half-equator. In the case of the sphere,
where we are in interested in {\it compact} minimal hypersurfaces (the
spherical Bernstein problem introduced by Chern \cite{chern}),
this is easy: When we compose our harmonic Gauss map with a strictly
convex function, we obtain a subharmonic function which on our compact
hypersurface then has to be constant, implying that the Gauss map
itself is constant, as desired. In fact, these results can also be obtained by the method
of Solomon \cite{so2}. In the Euclidean case, where we are
interested in complete minimal hypersurfaces, this is more
difficult. There, we need very precise a-priori estimates that can be
translated into a Liouville type theorem by a scaling argument. For
that purpose, unfortunately, we need to impose some additional
restrictions on the geometry of our minimal hypersurface. In
particular, we need a condition on the volume growth of balls as a
function of their radii, and we need a Poincar\'e
inequality. Fortunately, however, these assumptions are known to be
satisfied in a number of important and interesting cases, but the
final answers do not yet seem to be known (we are grateful to Neshan
Wickramasekera for some useful information in this regard, including
some description of his still unpublished work).

Let us finally try to put our results into the perspective of the
Bernstein problem (our survey will be rather incomplete, however; see \cite{xin2} for a more detailed account). The original result of Bernstein that there is no other entire minimal
graph in $\R^3$, i.e., a minimal graph defined on the entire plane
$\R^2$, than an affine plane, has been extended by Simons \cite{Si} to such
entire minimal graphs in $\R^{n}$ for $n \le 7$ whereas Bombieri-de
Giorgi-Giusti \cite{b-g-g} constructed counterexamples in higher
dimensions. In fact, the Bernstein problem has been one of the central driving
forces of geometric measure theory which is concerned with area (or in
higher dimensions, volume) minimizing currents (see \cite{fe}). More generally, such Bernstein type results have been
obtained for complete stable minimal hypersurfaces, on the basis of curvature
estimates by Heinz \cite{hz} (in dimension 2), Schoen-Simon-Yau
\cite{s-s-y}, Simon \cite{sim1,sim2},
Ecker-Huisken \cite{e-h}, and others. Minimal graphs are automatically stable,
and so this approach applies to the original problem. Also, in contrast
to the counterexample of \cite{b-g-g}, Moser \cite{M} had shown that an entire
minimal graph in any dimension has to be affinely linear, provided its
slope is uniformly bounded. \cite{h-j-w} then introduced the method
of deriving Bernstein type theorems by showing that the Gauss map of a
minimal submanifold of $\R^n$ is constant, as explained above. In
particular, this method could generalize Moser's result. See also
\cite{so1} for a combination of the Gauss map with geometric measure
theory constructions. An important advantage of the method of
\cite{h-j-w} as compared to either the geometric measure theory
approach or the strategy of curvature estimates is that it naturally
extends to higher codimension, the only difference being that the
Gauss map now takes its values in a Grassmann variety whose geometry
is somewhat more complicated than the one of a sphere. Nevertheless,
the Gauss map is still harmonic by \cite{r-v}, and when one can derive
good enough a-priori estimates, one can again deduce a Liouville type
theorem and Bernstein type results, see \cite{h-j-w,j-x}. Therefore, the
strategy of the present paper can also be extended to higher
codimension, and we shall develop the necessary convex geometry of
Grassmannians in a sequel to this paper.

\small \parskip0.1mm\tableofcontents \normalsize\parskip3mm
\bigskip\bigskip

\Section{Convex functions and spherical Bernstein
problem}{Construction of convex functions and the spherical Bernstein problem}
\label{s1}

\subsection{Convex supporting sets}

Let $(M,g)$ be a smooth Riemannian manifold. A $C^2$-function $F$
is said to be strictly \textbf{convex} on an open subset $U$ of
$M$ if the Hessian form of $F$ is  positive definite at every
point of $U$, i.e.,
$$\Hess\ F(X,X)=\n_X\n_X F-(\n_X X)F>0 \qquad \mbox{for every nonzero }X\in TU.$$
(Here $\n$ denotes the Levi-Civita connection on $M$ induced by
$g$.) Equivalently, for any arc-length-parametrized
geodesic $\g$ lying in $U$, $F\circ \g$ is a strictly convex
function in the usual sense.

The notion of a \textbf{convex supporting} set was proposed in
\cite{Go}. A subset $U$ of $M$ is said to be convex supporting
if and only if any compact subset of $U$ has an open
neighborhood in $M$ on which there is defined a strictly convex
function $F$. For the sequel, it may be helpful to point out that this does not
require  $F$  be  defined  on all $U$, and in fact in the case we
shall be interested in below, there will be no strictly convex function on $U$.

A maximal open convex supporting set is one which is not
properly contained in any other open convex supporting set. Take
the ordinary 2-sphere equipped with the canonical metric as an
example. An open hemisphere is obviously a convex supporting set,
but it is not a maximal one. To obtain a maximal open convex
supporting domain on $S^2$, it suffices to remove half of a great
circle $\g$ joining north pole and south pole.  We will prove this
fact and its higher dimensional analogue and  construct a maximal open convex
supporting set on $S^n$ ($n\geq 2$) equipped with the canonical
metric.

\begin{lem}\label{l3}
Let $M$ be a Riemannian manifold, $A$ be a compact domain of $M$ and
$h$ be a non-negative $C^2$-function $|\n h|\neq 0$
everywhere on $A$. If there is a positive
constant $C$ such that
\begin{equation}\label{Hessh2}
\Hess\ h(Y,Y)\geq C|Y|^2
\end{equation}
for any $Y\in TA$ with $dh(Y)=0$, then there exists a positive constant $\la_0$,
only depending on $C$, $\sup_A |\Hess\ h|$ and $\inf_A |\n h|$, such
that
whenever $\la\geq \la_0$,
\begin{equation}\label{Hessh}
\Hess\ \big(\la^{-1}\exp(\la h)\big)(X,X)\geq \f{C}{2}|X|^2
\end{equation}
for any $X\in TA$.
\end{lem}

\begin{proof}

With $\nu:=\f{\n h}{|\n h|}$,  for any unit tangent
vector $X\in TA$, there exist $\a\in [-\f{\pi}{2},\f{\pi}{2}]$ and a
unit tangent vector $Y\in TA$ such that $dh(Y)=0$ and
$$X=\sin \a\ \nu+\cos\a\ Y.$$
With $c_0:=\sup_U |\Hess\ h|$, then
$$\Hess\ h(\nu,\nu)\geq -c_0\qquad \mbox{and }\big|\Hess\ h(\nu,Y)\big|\leq c_0.$$
Thereby
$$\aligned
\Hess\ h(X,X)&=\sin^2\a \Hess\ h(\nu,\nu)+2\sin\a\cos\a \Hess\ h(\nu,Y)+\cos^2\a\Hess\ h(Y,Y)\\
&\geq -c_0\sin^2\a-2c_0|\sin\a\cos\a|+C\cos^2\a\\
&\geq -c_0\sin^2\a-\f{C}{2}\cos^2\a-2c_0^2C^{-1}\sin^2\a+C\cos^2\a\\
&=\f{C}{2}\cos^2\a-(c_0+2c_0^2C^{-1})\sin^2\a.
\endaligned$$
Denote $c_1:=\inf_{U}|\n h|$ and take
$$\la_0=c_1^{-2}\Big(\f{C}{2}+c_0+2c_0^2C^{-1}\Big),$$
then
$$\aligned
\Hess\big(\la^{-1}\exp(\la h)\big)(X,X)&=\exp(\la h)(\Hess\ h+\la dh\otimes dh)(X,X)\\
&\geq \Hess\ h(X,X)+\la \big(dh(X)\big)^2\\
&\geq \f{C}{2}\cos^2\a-(c_0+2c_0^2C^{-1})\sin^2\a+\la|\n h|^2\sin^2\a\\
&\geq \f{C}{2}\cos^2\a-(c_0+2c_0^2C^{-1})\sin^2\a+\la c_1^2\sin^2\a\\
&\geq \f{C}{2}\endaligned$$ whenever $\la\geq \la_0$ and
(\ref{Hessh}) follows.
\end{proof}

\medskip

\subsection{Maximal convex supporting subsets of $S^n$}
We work on the standard Euclidean sphere $S^n \subset
\mathbb{R}^{n+1}$ with its metric $g$.\\
We consider  a closed half hemisphere of $S^n$ of codimension 1, that
is, half an equator,
\begin{equation}\label{S+}
\ol{S}^{n-1}_+:= \{(x_1,x_2,\cdots,x_{n+1})\in S^n: x_1=0,x_2\geq
0\}
\end{equation}
and put
\begin{equation}\label{V}
\Bbb{V}:=S^n\backslash \ol{S}^{n-1}_+.
\end{equation}
$\Bbb{V}$ is open and connected.
$\ol{S}^{n-1}_+$ consists of the geodesics joining
$x_0=(0,1,0,\cdots,0)\in \ol{S}^{n-1}_+$ with the points in
$$S^{n-2}=\{(x_1,x_2,\cdots,x_{n+1})\in S^n:
x_1=x_2=0\},$$ which is a totally geodesic submanifold of $S^n$
with codimension $2$. When $n=2,$
$S^{n-2}=S^0=\{(0,0,1),(0,0,-1)\}$ and hence $\ol{S}^{1}_+$ is simply
the shortest geodesic joining $(0,0,1)$(north pole) and
$(0,0,-1)$(south pole) passing through $x_0=(0,1,0)$.

We start with some simple and well-known computations and consider  the projection  $\pi$ from $S^n$ onto $\bar{\Bbb{D}}^2$ (2-dimensional closed unit disk):
$$\pi:\ S^n\ra \bar{\Bbb{D}}^2\qquad (x_1,\cdots,x_{n+1})\mapsto (x_1,x_2).$$
Then $x\in \Bbb{V}$ if and only if $\pi(x)$ is contained in the
domain  obtained by removing the radius connecting (0,0) and
(0,1) from the closed unit disk. Hence for any $x\in
\Bbb{V}$, there exist a unique $v\in (0,1]$ and a unique
$\varphi\in (0,2\pi)$ such that
\begin{equation}\label{angle}
\pi(x)=(v\sin \varphi, v\cos\varphi).
\end{equation}
$v$ and $\varphi$ can be considered as smooth functions on $\Bbb{V}$.

Put $y_0:=(1,0,\cdots,0)$ and let $\rho$ be the distance function from $y_0$, then by  spherical geometry,
$x_1=\cos\rho$. It is well-known that
\begin{equation} \label{Hess}
\Hess\ \rho=\cot\rho(g-d\rho\otimes d\rho);
\end{equation}
hence
\begin{equation}\label{Hessx}\aligned
\Hess\ x_1&=\Hess \cos\rho=-\sin\rho\ \Hess\ \rho-\cos\rho\ d\rho\otimes d\rho\\
&=-\cos\rho (g-d\rho\otimes d\rho)-\cos\rho\ d\rho\otimes d\rho\\
&=-\cos\rho\ g=-x_1\ g.
\endaligned
\end{equation}
Similarly
\begin{equation}\label{Hessx2}
\Hess\ x_2=-x_2\ g.
\end{equation}
By (\ref{angle}),
\begin{equation}\label{v}
v^2=x_1^2+x_2^2.
\end{equation}
and
\begin{equation}\label{dx}\aligned
dx_1&=\sin\varphi d v+v\cos\varphi d\varphi=\sin\varphi dv+x_2 d\varphi,\\
dx_2&=\cos\varphi d v-v\sin\varphi d\varphi=\cos\varphi dv-x_1 d\varphi.
\endaligned
\end{equation}
Combining with (\ref{v}) and (\ref{dx}) yields
\begin{equation*}\aligned
&2v\Hess\ v+2dv\otimes dv
=\Hess\ v^2=\Hess (x_1^2+x_2^2)\\
&\qquad =2x_1\Hess\ x_1+2x_2\Hess\ x_2+2dx_1\otimes dx_1+2dx_2\otimes dx_2\\
&\qquad=-2v^2\ g+2dv\otimes dv+2v^2d\varphi\otimes d\varphi,
\endaligned
\end{equation*}
which tells us
\begin{equation}\label{Hessv}
\Hess\ v=-v\ g+v\ d\varphi\otimes d\varphi.
\end{equation}
Furthermore, (\ref{Hessx}), (\ref{Hessv}) and (\ref{angle}) tell us
$$\aligned
-x_1\ g &=\Hess\ x_1\\
&=v\cos\varphi\Hess \varphi+\sin\varphi\Hess\ v-x_1d\varphi\otimes d\varphi+\cos\varphi(d\varphi\otimes dv+dv\otimes d\varphi)\\
&=v\cos\varphi\Hess \varphi-x_1\ g+x_1 d\varphi\otimes d\varphi-x_1d\varphi\otimes d\varphi+\cos\varphi(d\varphi\otimes dv+dv\otimes d\varphi)\\
&=v\cos\varphi\Hess \varphi-x_1\ g+\cos\varphi(d\varphi\otimes
dv+dv\otimes d\varphi),
\endaligned$$
i.e.
$$v\cos\varphi\Hess\ \varphi=-\cos\varphi(d\varphi\otimes
dv+dv\otimes d\varphi).$$
Similarly, we have
$$v\sin\varphi\Hess\ \varphi=-\sin\varphi(d\varphi\otimes
dv+dv\otimes d\varphi).$$ We then have
\begin{equation}\label{Hessphi}
\Hess\ \varphi=-v^{-1}(d\varphi\otimes dv+dv\otimes d\varphi).
\end{equation}

Let $K$ be a compact subset of $\Bbb{V}.$ Define a function
\begin{equation}
\phi=\varphi+f(v)
\end{equation}
on $K,$ where $f$ is to be chosen. A straightforward calculation
shows that
\begin{equation}\aligned
\Hess\ \phi&=\Hess\ \varphi+f'(v)\Hess\ v+f''(v)dv\otimes dv\\
           &=-vf'(v)g+vf'(v)d\varphi\otimes d\varphi+f''(v)dv\otimes dv\\
           &\quad -v^{-1}(d\varphi\otimes dv+dv\otimes d\varphi).
\endaligned
\end{equation}
Obviously $d\phi=d\varphi+f'(v)dv\neq 0$, and for every $X$ such that $d\phi(X)=0$, we have
$d\varphi(X)=-f'(v)dv(X)$ and furthermore
\begin{equation}\aligned
\Hess\ \phi(X,X)&=-vf'(v)\lan X,X\ran+vf'(v)d\varphi(X)^2+f''(v)dv(X)^2-2v^{-1}d\varphi(X)dv(X)\\
                &=-vf'(v)\lan X,X\ran+\big(vf'(v)^3+f''(v)+2v^{-1}f'(v)\big)dv(X)^2.
                \endaligned
\end{equation}
By the compactness of $K$, there exists a constant $c\in (0,1)$,
such that $v>c$ on $K$. Hence the function
\begin{equation}
f=\arcsin\big(\f{c}{v}\big)
\end{equation}
is well-defined on $K$. By a straightforward computation, we obtain
$$\aligned
f'(v)&=-cv^{-1}(v^2-c^2)^{-\f{1}{2}}<0\\
f''(v)&=cv^{-2}(v^2-c^2)^{-\f{1}{2}}+c(v^2-c^2)^{-\f{3}{2}}
\endaligned$$
and moreover
$$\aligned
vf'(v)^3&+f''(v)+2v^{-1}f'(v)\\
=&-c^3v^{-2}(v^2-c^2)^{-\f{3}{2}}+cv^{-2}(v^2-c^2)^{-\f{1}{2}}+c(v^2-c^2)^{-\f{3}{2}}\\
&-2cv^{-2}(v^2-c^2)^{-\f{1}{2}} =0.\endaligned$$ Therefore
$\Hess\ \phi(X,X)>0$ for every $d\phi(X)=0$ and $|X|=1$. The
compactness of $K$ implies that we can find a positive constant
$C$ satisfying (\ref{Hessh2}). Then by  Lemma \ref{l3} we can find
$\la$ large enough so that
\begin{equation}
F=\la^{-1}\exp(\la \phi)
\end{equation}
is strictly convex on $K$. Since $K$ is arbitrary, we conclude that $\Bbb{V}$ is a convex supporting subset of $S^n$.

\begin{thm}\label{t1}

$\Bbb{V}=S^n\backslash \ol{S}^{n-1}_+$ is a maximal open convex
supporting subset of $S^n$.

\end{thm}

\begin{proof}

It remains to show that $\Bbb{V}$ is maximal. Let $\Bbb{U}\supset
\Bbb{V}$ be another open convex supporting subset of $S^n$. If there
exist $\th\in (0,\f{\pi}{2}]$ and $y\in S^{n-2}$, such that
$(0,\sin\th,y\cos\th)\in \Bbb{U}$, then a closed geodesic of $S^n$
defined by
\begin{equation}
\g: t\in \R\mapsto (\sin t,\cos t\sin\th,y\cos t\cos\th)
\end{equation}
lies in $\Bbb{U}$. (It is easily-seen that $|\dot{\g}|=1$ and
$$d\big(\g(t_0),\g(t_0+t)\big)=\arccos \big(\g(t_0),\g(t_0+t)\big)=t$$
whenever $t\in [0,\pi]$, hence $\g$ is a geodesic.) Since $\Bbb{U}$
is convex supporting, there exist an open neighborhood $U$ of
$\mbox{Im}(\g)$ and a strictly convex function $F$ on $U$. Hence
$$\f{d^2}{dt^2}(F\circ \g)=\Hess\ F(\dot{\g},\dot{\g})>0.$$
But on the other hand, since $F\circ \g$ is periodic, $F\circ \g$
takes its maximum at some point $t_0\in [0,2\pi]$; at $t_0$,
$\f{d^2}{dt^2}(F\circ \g)\leq 0$; which causes a contradiction.
Therefore $(0,\sin\th,y\cos\th)\notin \Bbb{U}$ whenever $\th\in
(0,\f{\pi}{2}]$. The openness of $\Bbb{U}$ yields $(0,0,y)\notin
\Bbb{U}$ whenever $y\in S^{n-2}$. Hence $\Bbb{U}=\Bbb{V}$ and we
complete the proof.

\end{proof}

The following {\bf Remark} may be helpful for the geometric
intuition:
\begin{itemize}
\item The functions that we construct in general do not have convex
  level sets. Only the intersections of their level sets with the
  compact subset $K$ of  $\Bbb{V}$ under consideration have to be
  convex.  That is, their level sets may leave $K$, become concave
  outside $K$, then enter $K$ again as convex sets and leave it again
  as concave sets.
\end{itemize}

The following two {\bf Observations} will be useful below:
\begin{enumerate}
\item[2.1]
\label{r1}
Let $F$ be a convex function on an arbitrary
compact set $K\subset \Bbb{V}$, and $T$ be an isometry of $S^n$ onto
itself, then obviously $F\circ T^{-1}$ is a convex function on
$T(K)$. Therefore $\Bbb{U}=T(\Bbb{V})=S^n\backslash T(\ol{S}^{n-1}_+)$ is also a maximal
open convex supporting subset of $S^n$. Here $T(\ol{S}^{n-1}_+)$ can be characterized by
$$T(\ol{S}^{n-1}_+)=\{x\in S^n: (x,e_1)=0\mbox{ and }(x,e_2)\geq 0\},$$
where $e_1,e_2$ are two orthogonal vectors on $S^n$. In the
sequel, $\ol{S}^{n-1}_+$ will denote an arbitrary codimension
$1$ closed half hemisphere.

\item[2.2]\label{cor}

Denote $\Bbb{D}^m(r):=\{x\in \R^m: |x|<r\}$ and
$\Bbb{D}^m:=\Bbb{D}^m(1)$. Then we can define
\begin{eqnarray}
\nonumber
\chi:& (0,2\pi)\times
\Bbb{D}^{n-1}&\ra \Bbb{V}\\
\nonumber
&(\th,y)&\mapsto (\sqrt{1-|y|^2}\sin\th,\sqrt{1-|y|^2}\cos\th,y).
\end{eqnarray}
It is easy to check that $\chi$ is a diffeomorphism. Thus,
$\Bbb{V}$ is diffeomorphic to a convex subset of $\R^n$. This fact
will be
crucial in the estimates of the oscillation of weakly harmonic maps,
see Section \ref{s3}.
\end{enumerate}

\medskip
\subsection{Liouville type theorems for harmonic maps from compact  manifolds}

The following property of harmonic maps is well-known (see
e.g. \cite{J1}, Section 7.2.C).

\begin{lem}

Let $(M^m,g)$, $(N^n,h)$ be two Riemannian manifolds (not necessarily
complete) and $u$ be a harmonic map from $M$ to $N$. If on $N$ there exists a strictly convex function $F$, then $F\circ u$ is a
subharmonic function. Moreover, if there exists a positive constant
$K_0$ such that $\Hess\ F\geq K_0\ h$, then
\begin{equation}\label{la1}
\De (F\circ u)\geq K_0|du|^2.
\end{equation}
\end{lem}

When $M$ is compact and $F$ is  a strictly convex function on $N$, then the compactness of
$u(M)$ enables us to find a constant $K_0 >0$ such that $\Hess\ F\geq
K_0h$ on $u(M)$. Therefore (\ref{la1}) holds. Integrating both sides
of (\ref{la1}) yields
$$0=\int_M \De(F\circ u)*1\geq K_0\int_M |du|^2*1=2K_0E(u).$$
Hence $E(u)=0$, i.e. $u$ is a constant map. We arrive at the
following Liouville-type theorem.

\begin{pro}\label{p1}(see \cite{Go})
Let $(M,g)$ be a compact Riemannian manifold, $(N,h)$ be a
Riemannian manifold and $u$ be a harmonic map from $M$ to $N$. If
the image of $u$ is contained in a convex supporting set of $N$,
then $u$ has to be a constant map.

\end{pro}

Theorem \ref{t1}, Proposition \ref{p1} and Observation 2.1 imply:

\begin{thm}
Let $(M,g)$ be a compact Riemannian manifold, $u$ be a harmonic
map from $M$ to $S^n.$ If $u(M)\subset S^n\backslash
\ol{S}^{n-1}_+,$ then $u$ has to be a constant map.

\end{thm}

\medskip

\subsection{A spherical Bernstein theorem}

S. S. Chern \cite{chern} has raised the
spherical Bernstein conjecture: Is any imbedded
minimal  $(n-1)$-dimensional sphere in $S^n$ an equator? For
$n=3,$ this was affirmatively solved by Almgren \cite{a} and Calabi
\cite{ca}. The answer is negative in higher dimensions, however, by counterexamples due to Hsiang \cite{hs}. Since then, the spherical
Bernstein problem is understood as the question under what  conditions a compact
minimal hypersurface in $S^n$ has to be an equator. An important
result of Solomon
\cite{so1}  concerns this problem for compact minimal hypersurfaces with
vanishing first Betti number.

We now study this problem for compact minimal hypersurfaces of arbitrary
topological type.

Let $M\ra S^{m+p}\subset \R^{m+p+1}$ be an $m$-dimensional
submanifold in the sphere. For  $x\in M$, by
parallel translation in $\R^{m+p+1}$, we can move the normal space $N_x M$ of
$M$ in $S^{m+p}$  to the origin of $\R^{m+p+1}$. Thereby
we get a $p$-subspace of $\R^{m+p+1}$. This defines the
\textbf{normal Gauss map} $\g: M\ra \grs{p}{m+1}$.
Here $\grs{p}{m+1}$ is the Grassmannian manifold
of  $p$-subspaces of $\R^{m+p+1}$.
When $p=1$, $\grs{p}{m+1}$ is simply the $(m+1)$-dimensional sphere.

There is a natural isometry $\eta$ between $\grs{p}{m+1}$ and
$\grs{m+1}{p}$ which maps any $p$-subspace into its orthogonal
complementary $(m+1)$-subspace. The map $\g^*=\eta\circ \g$ maps
any point $x\in M$ into the $(m+1)$-subspace consisting of the
tangent space of $M$ at $x$ and the position vector of $x$.

As pointed out and utilized by J. Simons \cite{Si}, the properties of
the (minimal) submanifold $M$ in the
sphere are closely related to those of the cone $CM$ generated by
$M$. This cone is the
image under the map from $M\times [0,\infty)$ into $\R^{m+p+1}$
defined by $(x,t)\mapsto tx$, where $t\in [0,\infty)$ and $x\in M$.
$CM$ has a singularity $t=0$. To avoid the singularity at the origin, we consider
the truncated cone $CM_\ep$, which is the image of
$M\times(\ep,\infty)$ under the same map, for $\ep >0$. We have

\begin{pro}\label{p2}(\cite{xin1} p.64)
$CM_\ep$ has parallel mean curvature in $\R^{m+p+1}$ if and only if
$M$ is a minimal submanifold in $S^{m+p}$.
\end{pro}

There is a natural map from $\R^{m+p+1}-\{0\}$ to $S^{m+p}$ defined
by
$$\psi(x)=\f{x}{|x|}\qquad x\in \R^{m+p+1}-\{0\}.$$
Hence for a map $f_1$ from a submanifold $M\subset S^{m+p}$
into a Riemannian manifold $N$, we obtain a map $f$ from
$CM_\ep$ into $N$ defined by $f=f_1\circ \psi$, which is called the
\textbf{cone-like map} (see \cite{xin1} p.66). One computes that $f_1$ is  harmonic  if and
only if $f$ is  harmonic (see \cite{xin1} p.67).

By the definition, it is clear that the Gauss map $\g_c:
CM_\ep\ra \grs{m+1}{p}$ $x\mapsto T_x (CM_\ep)$ is a cone-like
map. We have already defined the normal Gauss map $\g: M\ra
\grs{p}{m+1}$ and $\g^*=\eta\circ\g: M\ra \grs{m+1}{p}$, where
$\eta: \grs{p}{m+1}\ra \grs{m+1}{p}$ is an isometry. Obviously
\begin{equation}
\g_c=\g^*\circ\psi.
\end{equation}
The well-known Ruh-Vilms Theorem (see \cite{r-v}) tells us that $CM_\ep$
has parallel mean curvature if and only if the Gauss map $\g_c$ is a
harmonic map, which holds if and only if the normal Gauss map $\g$
is a harmonic map. In conjunction with Proposition \ref{p2} we have

\begin{pro}(\cite{ch} \cite{is} \cite{xin1}p.67)\label{p3}
$M$ is a minimal submanifold in the sphere if and only if its normal
Gauss map $\g: M\ra \grs{p}{m+1}$ is a harmonic map.
\end{pro}

Combining Proposition \ref{p3} and Theorem \ref{t1}, we obtain the
following spherical Bernstein theorem:

\begin{thm}\label{t2}
Let $M$ be a compact minimal hypersurface in $S^n.$ If the image
under the normal Gauss map omits $\ol{S}^{n-1}_+,$  then $M$ has
to be an equator.
\end{thm}
\noindent
{\bf Remarks:}
\begin{itemize}
\item
Theorem \ref{t2} is  an improvement of Simons' extrinsic
rigidity theorem (see \cite{Si}).
\item
Solomon  \cite{so1} (see
also \cite{so2}) showed that under the additional assumption that the
first Betti number of $M$ vanishes, such a spherical Bernstein already
holds when the Gauss image omits a neighborhood of some totally
geodesic $S^{n-2}$. Without that topological assumption, however, there are easy
counterexamples, like the Clifford torus, and its higher dimensional
analogues, as described in \cite{so1}. In fact, Theorem
\ref{t2} can also be obtained by Solomon's method in \cite{so2}.
\end{itemize}
\bigskip\bigskip

\Section{Construction of a smooth family of convex functions on $S^n$}
{Construction of a smooth family of convex functions on $S^n$}
So far, we have constructed and utilized a single convex function on
the target of our Gauss maps in the sphere. For the general regularity
theory for harmonic maps that we now wish to develop and later utilize
for Bernstein type theorems, we need suitable families of convex
functions. Therefore, we need to refine and extend our preceding construction.

On $\R^n$, the squared distances from the points $x\in \R^n$
constitute a smooth family of strictly convex functions, or
expressed differently, for every $x\in \R^n$, we have a strictly
convex function that assumes its minimum at $x$. In this vein, we
now wish to construct a smooth family of strictly convex
functions on arbitrary compact set $K\subset \Bbb{V}$,
sufficiently many points in $K$ occur as the minimal points of
corresponding convex functions. To realize this, we need the
following lemmas concerning the relationship between convex
hypersurfaces and convex functions.

\medskip
\subsection{Convex functions and convex hypersurfaces}
For later reference, we recall some elementary facts.
\begin{defi}
Let $N$
be a hypersurface in the  Riemannian manifold $(M,g)$. If there is a unit normal vector field
$\nu$ on $N$ with $\lan B(X,X),\nu\ran<0$ for any
nonzero $X\in TN$ (where $B$ denotes the second fundamental form
of $N$), then we call $N$  a \textbf{convex hypersurface}, and
the direction in which $\nu$ points is called the
\textbf{direction of convexity}.

\end{defi}
The following result is well known and easy to verify:
\begin{lem}\label{l1}
Let   $\phi$
be a $C^2$-function on the  Riemannian manifold $(M,g)$ and $N=\{x\in M: \phi(x)=c\}$  a level
set of $\phi$. If $|\n\phi|\neq 0$ on $N$, then
$$\Hess\ \phi(X,X)>0\qquad \mbox{for every nonzero }X\in TN$$
if and only if $N$ is a convex hypersurface
 and the direction of convexity is the direction of increasing
$\phi$.

\end{lem}

Combining  Lemma \ref{l1} with Lemma \ref{l3}, we  obtain

\begin{lem}
Let $A$ be a compact domain in $M$.
If $\phi$ is a nonnegative function on $A$, every level set of $\phi$ is
a convex hypersurface, and the direction of convexity is the direction of
increasing $\phi$, then there exists $\la>0$ such that $\la^{-1}\exp(\la\phi)$
is convex on $A$.
\end{lem}

The following result is again well known and easy to prove:
\begin{lem}\label{l2}
For  $x_0 \in S^n$
and $c \in (0,1)$,
 the hypersurface
$$N_{x_0,c}=\{x\in S^n: (x,x_0)=c\}$$
where $(.,.)$ denotes the Euclidean scalar product,
is convex, and the direction of convexity is the direction of
decreasing $(x,x_0)$.

\end{lem}

\medskip

\subsection{Refined construction of convex functions}

We shall use the functions $v$ and $\varphi$ defined in
(\ref{angle}). Let $K$ be a compact subset of
$\Bbb{V}=S^n\backslash\ol{S}_+^{n-1},$ then there exists a
constant $c\in (0,\f{1}{3}]$ with
$$v\geq 3c$$
on $K$.
\begin{equation}\label{U}
U:=\{x=(x_1,x_2,\cdots,x_{n+1})\in \Bbb{V}:v=\sqrt{x_1^2+x_2^2}>2c\},
\end{equation}
is an open domain of $S^n$ and $K\subset U$.

\begin{thm}\label{p0}

For any compact subset $\Phi$ of $(0,2\pi)$, there exists a
smooth family of nonnegative, smooth functions
 $F(\cdot,\varphi_0)$ ($\varphi_0\in \Phi$)
on $U$, such that:

 (i) $F(\cdot,\varphi_0)$ is strictly convex on $K$;

(ii) $F(x,\varphi_0)=0$ if and only if
$x=x_{\varphi_0}:=(\sin\varphi_0,\cos\varphi_0,0,\cdots,0)$;

(iii) $F(x,\varphi_0)\leq 1$ (or $F(x,\varphi_0)<1$) if and only
if $(x,x_{\varphi_0})\geq \f{3}{2}c$ (or respectively,
$(x,x_{\varphi_0})>\f{3}{2}c$) and $|\varphi-\varphi_0|\leq \pi$.

\end{thm}

\begin{proof}

Let $f$ be a smooth function on $[0,\infty)$ satisfying
$$\left\{\aligned
f(t)&=0,\hskip0.7in  t\in \left[0, 1-\f{3}{2}c\right];\\ f(t)&=
t-1+c,\quad t\in
\left[1-\f{1}{2}c, \infty\right);\\
0&\leq f'\leq 1.  \endaligned\right.$$ Then we can define $H$ on
$U\times (0,2\pi)\times [0,\infty)$ by
\begin{equation}\label{H}
(x,\varphi_0,t)\mapsto \left\{\begin{array}{cc}
-v\cos(\varphi-\varphi_0+f(t))+f(t)-t+1 & \mbox{if }\varphi\leq \varphi_0,\\
-v\cos(\varphi-\varphi_0-f(t))+f(t)-t+1 & \mbox{if }\varphi>\varphi_0.
\end{array}\right.
\end{equation}

Now we fix $\varphi_0\in (0,2\pi)$ and denote
$H_{\varphi_0}(x,t):=H(x,\varphi_0,t)$. For arbitrary $x\in U$,
put
\begin{equation}
I_x:=\big\{t\in [0,\infty): \max\{0,|\varphi-\varphi_0|-\pi\}\leq
f(t)\leq |\varphi-\varphi_0|\big\}
\end{equation}
then obviously $I_x$ is a closed interval,
$I_x=:[m_x,M_x]$. If $\varphi\leq \varphi_0$ and $t\in I_x$, then
\begin{equation}
\p_2
H_{\varphi_0}(x,t)=f'(t)\big(1+v\sin(\varphi-\varphi_0+f(t))\big)-1\leq
0
\end{equation}
and $\p_2 H_{\varphi_0}(x,t)=0$ if and only if $f'(t)=1$ and
$f(t)=|\varphi-\varphi_0|-\pi$ or $|\varphi-\varphi_0|$; which
implies $t=m_x$ or $M_x$. Hence
\begin{equation}\label{decrease}
\p_2 H_{\varphi_0}(x,\ \cdot\ )<0 \qquad \mbox{on }(m_x,M_x).
\end{equation}
Similarly (\ref{decrease}) holds when $\varphi>\varphi_0$.

It is easily seen that when $|\varphi-\varphi_0|\leq \pi$,
\begin{equation}\label{zero1}
H_{\varphi_0}(x,m_x)=H_{\varphi_0}(x,0)=-v\cos(\varphi-\varphi_0)+1\geq 0.\\
\end{equation}
Since $f'\leq 1$, $f(t)-t$ is a decreasing function. So
$f(t)-t\geq \lim_{t\ra \infty}\big(f(t)-t\big)=-1+c.$ Moreover,
when $|\varphi-\varphi_0|>\pi$ we have
$f(m_x)=|\varphi-\varphi_0|-\pi$ and by (\ref{H})
\begin{equation}\label{zero2}
H_{\varphi_0}(x,m_x)=v+f(m_x)-m_x+1\geq v+c>0.
\end{equation}
By the definition of $M_x$, $f$ cannot be identically zero on any
neighborhood of $M_x$; hence $M_x\geq 1-\f{3}{2}c$ and moreover
$f(M_x)-M_x\leq f(1-\f{3}{2}c)-(1-\f{3}{2}c)=-1+\f{3}{2}c$.
Therefore
\begin{equation}\label{max}
H_{\varphi_0}(x,M_x)=-v+f(M_x)-M_x+1<-2c-1+\f{3}{2}c+1=-\f{1}{2}c<0.
\end{equation}
By (\ref{decrease})-(\ref{max}), for each $x\in U$, there exists a
unique $\psi=\psi(x,\varphi_0)\in [m_x,M_x)$, such that
\begin{equation}
H(x,\varphi_0,\psi(x,\varphi_0))=H_{\varphi_0}(x,\psi(x,\varphi_0))=0.
\end{equation}

Denote
$$\Om:=\{(x,\varphi_0)\in U\times (0,2\pi):\varphi\neq \varphi_0\},$$
then $H$ is obviously smooth on $\Om\times [0,\infty)$. The
 implicit function theorem implies that $\psi$ is smooth
on $\Om$. To show the smoothness of $\psi$, it remains to prove that
$\psi$ is smooth on $\{(x,\varphi_0)\in U\times
(0,2\pi):\varphi=\varphi_0\}$. Denote
$$\Om_0:=\{(x,\varphi_0)\in U\times (0,2\pi): (x,x_{\varphi_0})\geq \f{3}{2}c\mbox{ and }|\varphi-\varphi_0|\leq \pi\},$$
then for every $(x,\varphi_0)\in \Om_0$,
$1-(x,x_{\varphi_0})=1-v\cos(\varphi-\varphi_0)\leq 1-\f{3}{2}c$
and hence $f\big(1-(x,x_{\varphi_0})\big)=0$; which implies
$1-(x,x_{\varphi_0})\in [m_x,M_x)$ and
$$H_{\varphi_0}(x,1-(x,x_{\varphi_0}))=-v\cos(\varphi-\varphi_0)-(1-(x,x_{\varphi_0}))+1=0.$$
Therefore
\begin{equation}\label{value}
\psi(x,\varphi_0)=1-(x,x_{\varphi_0})=1-\cos\rho(x)\qquad \forall
(x,\varphi_0)\in \Om_0
\end{equation}
(where $\rho$ denotes the distance from $x_{\varphi_0}$   on $S^n$), and
$\psi$ is obviously smooth on the interior of $\Om_0$. From
(\ref{U}) it is easily seen that $\{(x,\varphi_0)\in U\times
(0,2\pi):\varphi=\varphi_0\}\subset \mbox{int}(\Om_0)$, which yields the smoothness of $\psi$.

For  $\varphi_0\in (0,2\pi)$, put
\begin{equation}
V_{\varphi_0}:=\{x\in U:(x,x_{\varphi_0})\geq \f{3}{2}c\mbox{ and
}|\varphi-\varphi_0|\leq \pi\}.
\end{equation}
Then by (\ref{value}) and (\ref{Hess}), on $V_{\varphi_0}$,
\begin{equation}\label{Hess1}\aligned
\Hess\ \psi(\cdot,\varphi_0)&=\sin\rho\ \Hess\ \rho+\cos\rho\ d\rho\otimes d\rho\\
&=\cos\rho(g-d\rho\otimes d\rho)+\cos\rho\ d\rho\otimes d\rho\\
&=(1-\psi) g\geq \f{3}{2}c\ g;
\endaligned
\end{equation}
i.e. $\psi(\cdot,\varphi_0)$ is strictly convex on $V_{\varphi_0}$.

From (\ref{value}) it is easily seen that
$\psi(\cdot,\varphi_0)\leq 1-\f{3}{2}c$ on $V_{\varphi_0}$. On
the other hand, for arbitrary $x\in U\backslash V_{\varphi_0}$,
one of the following two cases must occur: (I)
$|\varphi-\varphi_0|>\pi$; (II) $|\varphi-\varphi_0|\leq \pi$ and
$(x,x_{\varphi_0})<\f{3}{2}c$.

If case (I) holds, then $\psi(x,\varphi_0)\geq m_x>1-\f{3}{2}c$
by $f(t)>0$;  in the second case, since
$$H_{\varphi_0}(x,1-\f{3}{2}c)=-(x,x_{\varphi_0})-(1-\f{3}{2}c)+1=-(x,x_{\varphi_0})+\f{3}{2}c>0$$
and by the monotonicity of $H_{\varphi_0}$ with respect to the $t$
variable (see (\ref{decrease})), we also have
$\psi(x,\varphi_0)>1-\f{3}{2}c$. Therefore
\begin{equation}\label{value2}
V_{\varphi_0}=\{x\in U: \psi(x,\varphi_0)\leq 1-\f{3}{2}c\}.
\end{equation}
Similarly
\begin{equation}\label{value3}
\mbox{int}(V_{\varphi_0})=\{x\in U: \psi(x,\varphi_0)<1-\f{3}{2}c\}.
\end{equation}

For each $t_0\geq 1-\f{3}{2}c$ and $x\in U$ satisfying
$\varphi<\varphi_0$, $\psi(x,\varphi_0)=t_0$ if and only if
$$0=H_{\varphi_0}(x,t_0)=-v\cos(\varphi-\varphi_0+f(t_0))+f(t_0)-t_0+1,$$
i.e.
$$(x,y(t_0))=f(t_0)-t_0+1,$$
where
$$y(t_0)=(\sin(\varphi_0-f(t_0)),\cos(\varphi_0-f(t_0)),0,\cdots,0).$$
Since $f'\leq 1$, $t\mapsto f(t)-t+1$ is a decreasing function,
hence

$$\f{3}{2}c\ge f(t_0)-t_0+1\geq \lim_{t\ra +\infty}(f(t)-t+1)=c.$$

By Lemma \ref{l2},
$$N_{t_0,\varphi_0}^-\mathop{=}\limits^{def.}\{x\in U:\varphi<\varphi_0,\psi(x,\varphi_0)=t_0\}$$
is a convex hypersurface, and the direction of convexity is the
direction of decreasing the function $x\ra (x,y(t_0))$. By noting
that
$$H_{\varphi_0}(x,t_0)=-(x,y(t_0))+f(t_0)-t_0+1,$$
we have
$$\n_{\nu}H_{\varphi_0}(\cdot,t_0)=-\n_{\nu}(\cdot,y(t_0))>0,$$
where $\nu$ is the unit normal vector field on $N_{t_0,\varphi_0}^-$
pointing in the direction of convexity. Then
$\td{H}(x):=H_{\varphi_0}(x,\psi(x,\varphi_0))$ satisfies $\td{H}\equiv
0$ and at each $x\in N_{t_0,\varphi_0}^-$
$$0=\n_{\nu}\td{H}=\n_{\nu}H_{\varphi_0}(\cdot,t_0)+\p_2 H_{\varphi_0}\n_\nu\psi(\cdot,\varphi_0);$$
which implies $\n_\nu\psi(\cdot,\varphi_0)>0$ (since $\p_2
H_{\varphi_0}(x,\psi(x,\varphi_0))<0$). In other words, $|\n
\psi(\cdot,\varphi_0)|\neq 0$ on $N_{t_0,\varphi_0}^-$ and the
direction of convexity of $N_{t_0,\varphi_0}^-$ is the direction of
increasing $\psi(\cdot,\varphi_0)$. By Lemma \ref{l1}, $\Hess\
\psi(\cdot,\varphi_0)(X,X)>0$ for every nonzero $X\in
TN_{t_0,\varphi_0}^-$ such that $d\psi(\cdot,\varphi_0)(X)=0$.

Similarly, putting
$$N_{t_0,\varphi_0}^+:=\{x\in U:\varphi>\varphi_0,\psi(x,\varphi_0)=t_0\},$$
then $|\n \psi(\cdot,\varphi_0)|\neq 0$ on $N_{t_0,\varphi_0}^+$ and
$\Hess\ \psi(\cdot,\varphi_0)(X,X)>0$ for every nonzero $X\in
TN_{t_0,\varphi_0}^+$ such that $d\psi(\cdot,\varphi_0)(X)=0$.

By the compactness of $K$ and $\Phi\subset (0,2\pi)$, there are
positive constants $c_2, c_3$ and $c_4$, such that for all
$\varphi_0\in \Phi$,
$$\Hess\ \psi(\cdot,\varphi_0)(X,X)\geq c_2|X|^2$$
for every nonzero $X\in TK$ which is tangential to one of the level
sets of $\psi(\cdot,\varphi_0)$, and
$$|\Hess\ \psi(\cdot,\varphi_0)|\leq c_3,\qquad|\n \psi(\cdot,\varphi_0)|\geq c_4$$
on $K$.  Lemma \ref{l3} then implies that there exists $\la_0>0$
satisfying
\begin{equation}\label{Hess2}
\Hess (\la_0^{-1}\exp(\la_0\psi(\cdot,\varphi_0))\geq \f{1}{2}c_2\
g.
\end{equation}
Now we take
\begin{equation}
F(\cdot,\varphi_0):=\f{\exp(\la_0\psi(\cdot,\varphi_0))-1}{\exp\big(\la_0(1-\f{3}{2}c)\big)-1}.
\end{equation}
Then from (\ref{Hess2}), $F(\cdot,\varphi_0)$ is a strictly convex
function on $K$; while conclusions (ii) and (iii) in the Theorem
follow from (\ref{value}), (\ref{value2}) and (\ref{value3}),
respectively.
\end{proof}

{\bf Remark:}
The auxiliary function $f$ in the above proof can be easily
obtained from the  standard bump functions. We choose a
nonnegative smooth function $h$ on $\R$, whose supporting set is
$[0,1].$ Let
$$h_1(t):=\f{\int_0^t h}{\int_0^1 h},$$
then $0\leq h_1\leq 1$, $h_1|_{(-\infty,0]}\equiv 0$ and
$h_1|_{[1,+\infty)}\equiv 1.$ Define
$$h_2(t):=h_1^\be(\f{t}{c}),$$
where $\be>0$ to be chosen, then $0\leq h_2\leq 1$,
$h_2|_{(-\infty,0]}\equiv 0$, and $h_2|_{[c,+\infty)}\equiv 1.$
Note that $\a\in (0,+\infty) \mapsto \int_0^c h_1^\a(\f{t}{c})$
is a strictly decreasing function, which converges to $0$ as
$\a\ra +\infty$ and converges to $c$ as $\a\ra 0.$ It enables us
to find $\be>0$, such that
$$\int_0^c h_2=\f{1}{2}c;$$
Then
$$f(t):=\int_0^{t-1+\f{3}{2}c}h_2$$
is the required function.

\bigskip\bigskip

\Section{Some properties of weakly harmonic maps}{Some properties of weakly harmonic maps}

Let $(M^m,g)$ and $(N^n,h)$ be Riemannian manifolds, not necessarily complete. Here and in the sequel, we denote
by $\{e_1,\cdots,e_m\}$ a local orthonormal frame field on $M$
and by $\{f_1,\cdots,f_n\}$ a local orthonormal frame field on
$N$. We use the summation convention with the index ranges
$$1\leq \a,\be\leq m,\qquad 1\leq i,j\leq n.$$

$u\in H_{loc}^{1,2}(M,N)$ is called a \textbf{weakly
harmonic map} if it is a critical point of the energy functional $E$;
i.e.
\begin{equation}
\f{d}{dt}\Big|_{t=0}E(\exp_u(t\xi))=0.
\end{equation}
for all compactly supported bounded sections $\xi$ of
$u^{-1}TN$ of class $H^{1,2}$, where $u^{-1}TN$ denotes the pull-back bundle of $TN$ (see \cite{J1} p.452).
A straightforward calculation shows
\begin{equation}\label{weak}
\int_M \lan du(e_\a),\n_{e_\a}\xi\ran*1=0.
\end{equation}
Here $\n$ is the
connection on $u^{-1}TN$ induced by the Levi-Civita
connections of $M$ and $N$.

Suppose $\Om$ is an open domain of $M$ and $K$ is a compact domain of $N$, such that $u(\Om)\subset K$
and there is a smooth and strictly convex function $F$ on $K$, i.e.
there exists a positive constant $K_0$ such that $\Hess\ F\geq K_0
h$.

Let $\eta$ be a non-negative smooth function on $\Om$ with compact
support. Put
\begin{equation}
\xi(y):=\eta(y)\n^N F\big(u(y)\big).
\end{equation}
Then (\ref{weak}) tells us
\begin{equation}\label{har}\aligned
0&=\int_\Om \lan du(e_\a),\n_{e_\a}\xi\ran*1\\
&=\int_\Om \lan du(e_\a),(\n_{e_\a}\eta)\n^N F\big(u(y)\big)\ran *1+\int_\Om \lan du(e_\a),\eta\n_{e_\a}\n^N F\big(u(y)\big)\ran*1\\
&=I+II.
\endaligned
\end{equation}
$f:=F\circ u$ then satisfies $\n_{e_\a} f=\n_{u_*e_\a}^N F$, hence
\begin{equation}\label{har1}
I=\int_\Om \n\eta\cdot \n f\ *1.
\end{equation}
Without loss of generality, one can assume $\n^N f_i=0$ for every
$1\leq i\leq n$ at the considered point, then
\begin{equation}\aligned
\n_{e_\a}\n^N F\big(u(y)\big)&=\n_{e_\a}\big((\n_{f_i}^N
F)f_i\big)=(\n_{u_*e_\a}^N\n_{f_i}^N F)f_i=\Hess\ F(u_* e_\a,f_i)f_i
\endaligned
\end{equation}
and moreover
\begin{equation}\label{har2}\aligned
II&=\int_\Om \eta\lan du(e_\a),\n_{e_\a}\n^N F\big(u(y)\big)\ran*1=\int_\Om \eta \Hess\ F(u_*e_\a,f_i)\lan f_i,u_*e_\a\ran*1\\
&=\int_\Om \eta \Hess\ F(u_*e_\a,u_*e_\a)*1\geq K_0\int_\Om \eta |u_*e_\a|^2*1\\
&=K_0\int_\Om \eta|du|^2*1.
\endaligned
\end{equation}
Substituting (\ref{har1}) and (\ref{har2}) into (\ref{har}) yields
\begin{equation}\label{subhar}
K_0\int_\Om \eta|du|^2*1\leq -\int_\Om \n\eta\cdot \n f*1.
\end{equation}
It says that $f=F\circ u$ is a subharmonic function in the weak
sense.

\subsection{Additional assumptions}

In the following, we shall assume that $(M,g)$ satisfies 3
additional conditions:

(D) There is a distance function $d$ on $M$ (which is not
necessary induced from the Riemannian metric on $M$), and the metric
topology induced by $d$ is equivalent to the Riemannian topology of
$M$; moreover, for each $y_1,y_2\in M$, $d(y_1,y_2)\leq
r(y_1,y_2)$, where $r(\cdot,\cdot)$ is the distance function of
$M$ with respect to the Riemannian metric.

(V) \textbf{Doubling property}: Let $B_R(y)$ be the ball centered
at $y$ of radius $R$ given by the distance $d$, denote by $V(y,R)$
the volume of $B_R(y)$, then there are $R_0\in (0,\infty]$ and a
positive constant $K_1$ independent of $y$ and $R$, such that
\begin{equation}
\label{double}
V(y,2R)\leq K_1\ V(y,R)\qquad \mbox{whenever }R\leq R_0.
\end{equation}

(P) \textbf{Neumann-Poincar\'{e} inequality}: For arbitrary $y\in
M$ and $R>0$ satisfying $B_R(y)\subset\subset M$, the following
inequality holds
\begin{equation}\label{str-p}
\int_{B_{ R}( y)}|v-\bar{v}_{B_{ R}(y)}|^2*1\leq
K_2R^2\int_{B_R(y)}|\n v|^2 *1
\end{equation}
where
$$\bar{v}_{B_{ R}(y)}:=\f{\int_{B_{ R}(y)}v*1}{V(y,R)}$$
is the average value of $v$ on $B_{ R}(y)$, and $K_2$ is a positive
constant not depending on $y$ and $R$.

We say that  the manifold $M$ satisfies the DVP-condition if it satisfies
these 3 conditions.

{\bf Remarks:}
\begin{enumerate}
\item[4.1]
\label{r2}
Condition (D) implies $|\n d(\cdot,y)|\leq 1$ for each $y\in M$.
Hence it is easy to construct a cut-off function $\eta$ on
$B_R(y)$ satisfying
$$0\leq \eta\leq 1,\ \eta|_{B_r(y)}\equiv 1,\mbox{and } |\n \eta|\leq c_0(R-r)^{-1}$$
by letting $\eta=\varphi\big(d(\cdot,y)\big)$,
where $\varphi$ is a smooth function on $[0,\infty)$,
such that $0\leq \varphi\leq 1$, $\varphi|_{[0,r]}=1$, $\varphi|_{[R,\infty)}=0$, and
$|\varphi'|\leq \f{c_0}{R-r}$.

\item[4.2]
Put
\begin{equation}
\nu_0:=\f{\log K_1}{\log 2}.
\end{equation}
For arbitrary $0<r<R\leq R_0$, we consider the integer $k$ such that
$2^{k-1}<\f{R}{r}\leq 2^k$; from the doubling property it follows
that
\begin{equation}\label{con1}
V(x,R)\leq V(x,2^k r)\leq K_1^k V(x,r)<
K_1\big(\f{R}{r}\big)^{\nu_0}V(x,r).
\end{equation}

\item[4.3]
It is well-known that the Neumann-Poincar\'{e} inequality is closely
related to the eigenvalues of the Laplace-Beltrami operator with Neumann
boundary values. More precisely, let $\mu_2(\Om)$ be the second
eigenvalue of
\begin{equation}\aligned
\De v+\mu v=0\qquad& \mbox{in }\Om\\
\f{\p v}{\p n}=0\qquad &\mbox{on }\p\Om
\endaligned
\end{equation}
where $n$ denotes the outward normal vector field,
then $\mu_2(\Om)$ is characterized by
\begin{equation}
\mu_2(\Om)=\min_{\int_\Om v*1=0}\f{\int_\Om |\n v|^2*1}{\int_\Om
|v|^2*1}.
\end{equation}
Therefore Condition (P) is equivalent to
$\mu_2(B_R(y))\geq K_2^{-1}R^{-2}$.

The Neumann-Poincar\'{e} inequality is also related to Cheeger's \cite{c}  isoperimetric constant
\begin{equation}
h_N(\Om):=\inf_{A}\f{\mbox{Vol}(\p A\cap
\mbox{int}(\Om))}{\mbox{Vol}(A)}
\end{equation}
where $A$ stands for an open subset of $\Om$
satisfying $\mbox{Vol}(A)\leq \f{1}{2}\mbox{Vol}(\Om)$. Cheeger
proved
\begin{equation}
\mu_2(\Om)\geq \f{1}{4}h_N^2(\Om).
\end{equation}

\item[4.4]
(\ref{str-p}) is the strong form of the Poincar\'{e} inequality; in
contrast, the weak form of Poincar\'{e} is
\begin{equation}\label{weak-p}
\int_{B_{\k R}( y)}|v-\bar{v}_{B_{\k R}(y)}|^2*1\leq C
R^2\int_{B_R(y)}|\n v|^2 *1
\end{equation}
where $\k\in (0,1)$ and $C$ is a constant not depending on $y$ and
$R$. It follows from the work of D. Jerison \cite{je} that the
doubling property and (\ref{weak-p}) implies (\ref{str-p}). Hence
Condition (P) could be replaced by (\ref{weak-p}).

\item[4.5]
From the work of Saloff-Coste \cite{sc} and Biroli-Mosco
\cite{b-m2}, Conditions (V) and (P) imply the following
Sobolev-type inequality: for $y\in M$ and $R>0$ satisfying
$B_{2R}(y)\subset\subset M$,
\begin{equation}\label{sobolev}
\Big(\int_{B_R(y)}|v|^{\f{2\nu}{\nu-2}}*1\Big)^{\f{\nu-2}{2\nu}}\leq
K_3R\ V(y,R)^{-\f{1}{\nu}} \big(\int_{B_R(y)}|\n
v|^2*1\big)^{\f{1}{2}},
\end{equation}
where $v\in H_0^{1,2}(B_R(y))$, $\nu$ is a constant only depending
on $K_1$, which is greater or equal to $\nu_0=\f{\log K_1}{\log 2}$
and strictly greater than $2$; $R\leq R_0$ and $K_3$ is a positive
constant only depending on $K_1$ and $K_2$. With
\begin{equation}
\aint{\Om}v:=\f{\int_\Om v*1}{\V(\Om)},
\end{equation}
(\ref{sobolev}) is equivalent to
\begin{equation}\label{con2}
\Big(\aint{B_R(y)}|v|^{\f{2\nu}{\nu-2}}\Big)^{\f{\nu-2}{2\nu}}\leq
K_3R \big(\aint{B_R(y)}|\n v|^2\big)^{\f{1}{2}}.
\end{equation}

\end{enumerate}

\medskip
\subsection{Harnack inequality}

In the sequel, we shall make use of the following abbreviations:
Fix a point $y_0\in M$ and let
$B_R=B_R(y_0)\subset\subset M$ with $R\leq \f{1}{2}R_0$, then
$V(R):=V(y_0,R)$ and for arbitrary $v\in
L^\infty(B_R)$,
\begin{equation}\aligned
&v_{+,R}:=\sup_{B_R}v,\qquad
v_{-,R}:=\inf_{B_R}v,
\qquad \bar{v}_R :=\aint{B_R}v,\\
&|\bar
v|_{p,R}:=\Big(\aint{B_R}|v|^p\Big)^{\f{1}{p}}\qquad
p\in (-\infty,+\infty).
\endaligned
\end{equation}
It is easily seen that $p\mapsto |\bar v|_{p,R}$ is an increasing
function, $\lim_{p\ra +\infty}|\bar v|_{p,R}=|v|_{+,R}$ and
$\lim_{p\ra -\infty}|\bar v|_{p,R}=|v|_{-,R}$, if $|\bar
v|_{p,R}$ is well-defined.

\begin{lem}\label{l8}
Let $M$ be an $m$-dimensional Riemannian manifold satisfying
DVP-condition, then for any positive superharmonic
function $v$ on $B_R$ satisfying $R\leq \f{1}{2}R_0$ and
$B_{2R}\subset\subset M$, $p\in (0,\f{\nu}{2-\nu})$ and $\th\in
[\f{1}{2},1)$, we have the  estimate
\begin{equation}
|\bar v|_{p,\th R}\leq \g_1 v_{-,\th R}
\end{equation}
Here $\g_1$ is a positive constant only depending on $K_1,K_2,p$ and
$\th$.
\end{lem}

This Harnack inequality  follows from the work of
Moser\cite{M}, Bombieri-Giusti\cite{b-g}, Saloff-Coste\cite{sc}
and Biroli-Mosco\cite{b-m} as we shall now briefly describe. Firstly, the superharmonicity of $v$
implies a  reserve Poincar\'{e} inequality for $v^{k}$
for arbitrary $k<\f{1}{2}$ (see \cite{M} Lemma 4); with the aid
of the Sobolev inequality (\ref{con2}) and suitable cut-off functions
as described in Remark \ref{r2},  we can obtain
\begin{equation}\label{sup1}
v_{-,(1-\tau)R}\geq (c_1\tau^{-c_2})^{-\f{1}{q}}|\bar v|_{-q,R}
\end{equation}
for arbitrary $\tau\in (0,\f{1}{2}]$ and $q\in (0,\infty)$ by
Moser's iteration. Again using Moser's iteration repeatedly,
one can get
\begin{equation}\label{sup2}
|\bar v|_{q,(1-\tau)R}\leq
(c_{3}\tau^{-c_{4}})^{\f{1}{s}-\f{1}{q}}|\bar v|_{s,R}.
\end{equation}
Here $\tau\in (0,\f{1}{2}]$, $q\in (0,\f{\nu}{\nu-2})$, $s\in (0,(\f{2}{q}-\f{\nu-2}{\nu})^{-1}]$ and
$c_3, c_4$ are positive constants
depending only on $K_1,K_2$ and $q$. By the Neumann-Poincar\'{e} inequality, we arrive at
\begin{equation}\label{sup3}
\sup_{\tau\in [\tau_0,\f{1}{2}]}\inf_{k\in
\R}\aint{B_{(1-\tau)R}}\Big|\log\big(\f{v}{k}\big)\Big|\leq
c_5(\tau_0,K_1,K_2)
\end{equation}
as in \cite{b-g}, where $\tau_0\in (0,\f{1}{2})$. Combining with
(\ref{sup1}), (\ref{sup2}) and (\ref{sup3}), one can apply an
abstract John-Nirenberg inequality (\cite{b-g}, Theorem 4) to
obtain the result.

From Lemma \ref{l8}, we get  analogues of
Corollary 1 and Lemma 7 in \cite{J}:

\begin{cor}\label{l4}
Let $M$ be an $m$-dimensional Riemannian manifold satisfying
DVP-condition, $v$ be a subharmonic function on $B_R$
satisfying $R\leq \f{1}{2}R_0$ and $B_{2R}\subset\subset M$. Then
there exists a constant $\de_0\in (0,1)$, only depending on $K_1$
and $K_2$, such that
\begin{equation}\label{ineq6}
v_{+,\f{R}{2}}\leq (1-\de_0)v_{+,R}+\de_0 \bar v_{\f{R}{2}}.
\end{equation}

\end{cor}

\begin{proof}

For arbitrary $\ep>0$, $v_{+,R}-v+\ep$ is a positive superharmonic
function on $B_R$, then Lemma \ref{l8} implies
$$|v_{+,R}-v+\ep|_{1,\f{R}{2}}\leq \g_1(v_{+,R}-v+\ep)_{-,\f{R}{2}}$$
where $\g_1$ is a positive constant only depending on $K_1$ and
$K_2$. This is equivalent to $v_{+,R}-\bar v_{\f{R}{2}}+\ep\leq
\g_1(v_{+,R}-v_{+,\f{R}{2}}+\ep)$; letting $\ep\ra 0$ yields
$$v_{+,R}-\bar v_{\f{R}{2}}\leq \g_1(v_{+,R}-v_{+,\f{R}{2}}).$$
(\ref{ineq6}) follows  by
putting $\de_0=\f{1}{\g_1}$.
\end{proof}
The next result is proved as in \cite{J}:
\begin{cor}\label{c1}
Let $v$ be as in Corollary \ref{l4}, and
suppose $0<\ep<\f{1}{2}$. There exists $k\in \Bbb{N}$,
independent of $v$ and $\ep$, such that
\begin{equation}
v_{+,\ep^k R}\leq \ep^2 v_{+,R}+(1-\ep^2)\bar v_{R'}
\end{equation}
for some $R'$ with $\ep^k R\leq R'\leq \f{R}{2}$ ($R'$ may depend on $v$ and $\ep$).
\end{cor}

\medskip

\subsection{Mollified Green function}

Obviously, $b:H_0^{1,2}(\Om)\times H_0^{1,2}(\Om)\ra \R$ defined
by
$$(\phi,\psi)\mapsto \int_{\Om}\n\phi\cdot\n\psi*1$$
is a bounded, positive definite bilinear form. For arbitrary $y\in \Om$ and $\rho>0$
such that $B_\rho(y)\subset \Om$,
$$\phi\in H_0^{1,2}(\Om)\mapsto \aint{B_\rho(y)}\phi$$
is a bounded linear functional. By the  Lax-Milgram Theorem, there exists a unique function
$G^\rho(\cdot,y)\in H_0^{1,2}(\Om)$, such that for all $\phi\in H_0^{1,2}(\Om)$,
\begin{equation}\label{gr}
\int_{\Om}\n G^\rho(\cdot,y)\cdot \n\phi*1=\aint{B_\rho(y)}\phi.
\end{equation}
$G^\rho$ is called the \textbf{mollified Green function} with respect
to the Laplace-Beltrami operator on
$\Om$. We can follow \cite{g-w} and \cite{b-m} to obtain
estimates on $G^\rho$, to be used in the next paragraphs. (Riemannian
manifolds satisfying the DVP-
condition  are certain metric spaces  (homogeneous
spaces in the sense of  \cite{Co} Ch. III, Section 1) on which a
weak version of the Poincar\'{e} inequality holds, and $(u,v)\mapsto
\int_M \n u\cdot\n v*1$ is a Dirichlet form. Hence the
results in \cite{b-m} can be applied.)

\begin{lem}

Let $M$ be an $m$-dimensional Riemannian manifold satisfying DVP-condition, $R\in(0, \f{1}{3}R_0]$
satisfying $B_{3R}\subset\subset M$. Then
the mollified Green function $G^\rho$  on $B_R$ enjoys the following properties:
\begin{equation}\label{gr1}
G^\rho(\cdot,y)\leq C_1\f{R^2}{V(R)}\qquad \mbox{on }S_R:=B_R-\bar{B}_{\f{3R}{8}}
\end{equation}
and
\begin{equation}\label{gr2}
\int_{T_R}|\n G^\rho(\cdot,y)|^2*1\leq C_2\f{R^2}{V(R)}\qquad\mbox{on }T_R:=B_R-\bar{B}_{\f{R}{2}}
\end{equation}
for all $y\in B_{\f{R}{4}}$ and $\rho\leq \f{R}{8}$.
Here $C_1,C_2$ are  positive constants depending only on $K_1$ and $K_2$.

\end{lem}

\begin{proof}

$y\in B_{\f{R}{4}}$ implies $B_R\subset B_{\f{5R}{4}}(y)$; since
$B_{\f{5R}{2}}(y)\subset B_{\f{11R}{4}}\subset
B_{3R}\subset\subset M$, one can apply  (6.13)-(6.15) in
\cite{b-m} to obtain
\begin{equation}
\sup_{\p B_{\f{R}{8}}(y)}G_{B_{\f{5R}{4}}(y)}^\rho(\cdot,y)\leq c_6(K_1,K_2)\f{R^2}{V(y,\f{R}{8})}.
\end{equation}
Here $G_{B_{\f{5R}{4}}(y)}^\rho$ denotes the mollified Green
function on $B_{\f{5R}{4}}(y)$, which is harmonic on
$B_{\f{5R}{4}}(y)-B_\rho(y).$ Hence the maximal principle implies
for each $z\in S_R\subset B_{\f{5R}{4}}(y)-B_{\f{R}{8}}(y)$,
\begin{equation}\label{bm}
G_{B_{\f{5R}{4}}(y)}^\rho(z,y)\leq \sup_{\p B_{\f{R}{8}}(y)}G_{B_{\f{5R}{4}}(y)}^\rho(\cdot,y)
\leq c_6(K_1,K_2)\f{R^2}{V(y,\f{R}{8})}
\end{equation}
Noting that $G_{B_{\f{5R}{4}}(y)}^\rho$
and $G^\rho$ are all nonnegative (which can be seen by a simple truncation argument, see \cite{g-w}) and
$G_{B_{\f{5R}{4}}(y)}^\rho-G^\rho$ is harmonic on $B_R$, again using the maximal principle yields
$$G^\rho(z,y)-G_{B_{\f{5R}{4}}(y)}^\rho(z,y)\leq \sup_{\p B_R}\big[G^\rho(\cdot,y)-G_{B_{\f{5R}{4}}(y)}^\rho(\cdot,y)\big]\leq 0.$$
Hence (\ref{gr1}) immediately follows from (\ref{bm}) and
$$V(y,\f{R}{8})\geq K_1^{-4}V(y,\f{5R}{4})\geq K_1^{-4}V(R).$$

As in \cite{g-w}, we choose a cut-off function $\eta$ satisfying $\eta\equiv 1$ in $T_R$,
$\eta\equiv 0$ in $B_{\f{3R}{8}}$ and $|\n \eta|\leq \f{c_0}{R}$, and insert $G^\rho(\cdot,y)\eta^2$ into
(\ref{gr}). (\ref{gr2}) then follows from (\ref{gr1}).

\end{proof}

\begin{lem}\label{l6}
Let $M$ be an $m$-dimensional Riemannian manifold satisfying DVP-condition, $R\in(0, \f{1}{2}R_0]$
satisfying $B_{2R}\subset\subset M$. With
\begin{equation}\label{om}
\om^R:=\f{V(\f{R}{2})}{R^2}G^{\f{R}{2}}(\cdot,y_0),
\end{equation}
then
\begin{equation}\label{gr3}
\om^R\leq C_3\qquad \mbox{on }B_R
\end{equation}
and
\begin{equation}\label{gr4}
\om^R\geq C_4\qquad \mbox{on }B_{\f{R}{2}},
\end{equation}
where $C_3$ and $C_4$ are positive constants depending only on
$K_1$ and $K_2$, but not depending on $R$.
\end{lem}

\begin{proof}
(\ref{gr}) and (\ref{om}) imply that
\begin{equation}\label{om2}
\int_{B_R}\n \om^R\cdot \n\phi*1=\f{1}{R^2}\int_{B_{\f{R}{2}}}\phi*1=\int_{B_R}\f{1}{R^2}\mathbf{1}_{B_{\f{R}{2}}}\cdot \phi*1
\end{equation}
holds for every $\phi\in H_0^{1,2}(B_R)$. Then applying Theorem
4.1 in \cite{b-m}  yields
\begin{equation}
\sup_{B_R}\om^R\leq C_3R^2\sup_{B_R}\Big(\f{1}{R^2}\mathbf{1}_{B_{\f{R}{2}}}\Big)=C_3.
\end{equation}

By  (6.13)-(6.15) in \cite{b-m},
\begin{equation}
\inf_{\p B_{\f{R}{2}}}G^{\f{R}{2}}(\cdot,y_0)\geq C_4\f{R^2}{V(\f{R}{2})}.
\end{equation}
Since $G^{\f{R}{2}}(\cdot,y_0)$ is a superharmonic function, it
assumes its minimum in $\bar{B}_{\f{R}{2}}$ at the boundary;
therefore,
(\ref{gr4}) immediately follows.

\end{proof}

\medskip

\subsection{Telescoping lemma}

Based on (\ref{subhar}) and Corollary \ref{l4}, we can obtain a version of the
telescoping lemma of Giaquinta-Giusti \cite{g-g} and
Giaquinta-Hildebrandt \cite{g-h} as in \cite{J}.

\begin{lem}\label{l5}

Let $(M^m,g)$ be a Riemannian manifold satisfying
DVP-condition, $(N^n,h)$ be a Riemannian
manifold, $u\in H_{loc}^{1,2}(M,N)$ be a weakly harmonic map, $K$ be a
compact domain of $N$, and let there exist a smooth and strictly convex function $F$ on
$K$ such that $\Hess\ F\geq K_0\ h$. If there is $R_1\in (0,\f{1}{2}R_0]$, such that
$B_{2R_1}\subset\subset M$ and $u(B_{R_1})\subset K$, then there is a positive constant $C_5$, only depending on
$K_0,K_1$ and $K_2$, such that for arbitrary $R\leq R_1$
\begin{equation}\label{tele}
\f{R^2}{V(\f{R}{2})}\int_{B_{\f{R}{2}}}|du|^2*1\leq C_5(f_{+,R}-f_{+,\f{R}{2}}).
\end{equation}
Here $f=F\circ u$. Moreover, there exists a positive constant $C_6$,
only depending on $K_0,K_1,K_2$ and $\sup_K F-\inf_K F$,
with the property that for arbitrary $\ep>0$, we can find $R\in
[\exp(-C_6\ep^{-1})R_1,R_1]$ such that
\begin{equation}
\f{R^2}{V(\f{R}{2})}\int_{B_{\f{R}{2}}}|d u|^2*1\leq \ep.
\end{equation}
\end{lem}

The telescoping lemma will be so powerful for our purposes because it
does not require an energy bound on our weakly harmonic map. Instead,
the energy of $u$ is locally controlled by the oscillation of its
composition with the strictly convex function $F$, essentially via a
lower bound on the Hessian of $F$. (In general, when applying this
scheme, one will also need  an
upper bound for the gradient of the strictly convex function $F$ in order to relate the oscillation of
$F\circ u$ to the one of $u$ itself. In the situation of the present
paper, this will be implicitely contained in the geometry of the
sphere and therefore not come up as an issue.)

\begin{proof}
$v:=f-f_{+,R}$ satisfies $v\leq 0$. Choosing $(\om^R)^2 \in H_0^{1,2}(B_R)$ as a test function in (\ref{subhar})
($\om^R$ is defined in (\ref{om})), we obtain
$$\aligned
K_0\int_{B_{R}}|du|^2(\om^R)^2*1&\leq -\int_{B_{R}}\n(\om^R)^2\cdot \n v*1=-2\int_{B_{R}}\n\om^R\cdot \om^R\n v*1\\
&=-2\int_{B_{R}}\n\om^R\cdot\big(\n(\om^R v)-v\n \om^R\big)*1\leq -2\int_{B_{R}}  \n\om^R\cdot\n(\om^R v)*1\\
&=-\f{2}{R^2}\int_{B_{\f{R}{2}}}\om^R v*1\leq
-\f{2C_3}{R^2}\int_{B_{\f{R}{2}}}v*1.
\endaligned$$
Here we have used (\ref{om2}) and the pointwise estimates for $\om^R$ in Lemma \ref{l6}.
On the other hand,
$$
\int_{B_R}|du|^2(\om^R)^2*1\geq C_4^2\int_{B_{\f{R}{2}}}|du|^2*1.
$$
Hence
\begin{equation}\label{es1}\aligned
\int_{B_{\f{R}{2}}}|du|^2*1&\leq \f{2C_3}{K_0C_4^2R^2}\int_{B_{\f{R}{2}}}(f_{+,R}-f)*1\\
&=c_7(K_0,K_1,K_2)\f{V(\f{R}{2})}{R^2}(f_{+,R}-\bar f_{\f{R}{2}}).
\endaligned
\end{equation}
By Corollary \ref{l4}, $f_{+,R}-\bar f_{\f{R}{2}}\leq
\de_0^{-1}(f_{+,R}-f_{+,\f{R}{2}})$. Substituting it into
(\ref{es1}) yields (\ref{tele}).

For arbitrary $k\in \Bbb{N}$, (\ref{tele}) tells us
\begin{equation}\aligned
\sum_{i=0}^k \f{(2^{-i}R_1)^2}{V(2^{-i-1}R_1)}\int_{B_{2^{-i-1}R_1}}|du|^2*1&\leq C_5\sum_{i=0}^k (f_{+,2^{-i}R_1}-f_{+,2^{-i-1}R_1})\\
                                                        &=C_5(f_{+,R_1}-f_{+,2^{-k-1}R_1})\\
                                                        &\leq C_5(\sup_K F-\inf_K F)
\endaligned
\end{equation}
For arbitrary $\ep>0$, we take
\begin{equation}
k:=\big[C_5(\sup_K F-\inf_K F)\ep^{-1}\big].
\end{equation}
Here and in the sequel, $[x]$ denotes the greatest integer not larger
than
$x.$ Then we can find $j$ with $0\leq j\leq k,$ such that
\begin{equation}
\f{(2^{-j}R_1)^2}{V(2^{-j-1}R_1)}\int_{B_{2^{-j-1}R_1}}|du|^2*1\leq \f{1}{k+1}C_5(\sup_K
F-\inf_K F)\leq \ep.
\end{equation}
Since $2^{-j}\geq 2^{-k}\geq 2^{-C_5(\sup_K F-\inf_K
F)\ep^{-1}}=\exp\big(-(\log 2) C_5(\sup_K F-\inf_K F)\ep^{-1}\big)$,
it is sufficient to take $C_6= C_5\log 2(\sup_K F-\inf_K F)$.
\end{proof}
\bigskip\bigskip

\Section{Regularity of weakly harmonic maps and Liouville type
theorems}{Regularity of weakly harmonic maps and Liouville type
theorems}\label{s3}

\subsection{Pointwise estimates}

$M,N,u,K$ and $R_0$ are as in Lemma \ref{l5}. Now we assume that there exists $R_1\in
(0,\f{1}{3}R_0]$ with $B_{3R_1}\subset\subset M$ and $u(B_{R_1})\subset K$.
Let $H$ be a smooth function on $K$, $\eta$ be a non-negative
smooth function on $B_{R_1}$ with compact support and $\varphi$
be a $H^{1,2}$-function on $B_{R_1}$. Denoting
\begin{equation}
\xi(y):=\eta(y)\varphi(y)\n^N H(u(y)),
\end{equation}
then similar to (\ref{har})-(\ref{har2}), we have
\begin{equation}
\int_{B_{R_1}} \varphi\n\eta\cdot \n h*1+\int_{B_{R_1}} \eta\n\varphi\cdot \n
h*1+\int_{B_{R_1}} \eta\varphi\De h*1=0
\end{equation}
where $h=H\circ u$. It implies
\begin{equation}\label{har3}\aligned
\int_{B_{R_1}} \n\varphi\cdot \n(\eta h)*1=&-\int_{B_{R_1}} \varphi\n\eta\cdot \n h*1-\int_{B_{R_1}} \eta\varphi\De h*1\\
&+\int_{B_{R_1}}h\n\varphi\cdot\n\eta*1.
\endaligned
\end{equation}
For arbitrary $R\leq \f{1}{2}R_1$, we can take a cut-off function
$\eta$ with the support in the interior of $B_R$, $0\leq \eta\leq
1$, $\eta\equiv 1$ on $B_{\f{R}{2}}$ and $|\n\eta|\leq
\f{c_0}{R}$. For each $\rho\leq \f{R}{8}$, denote by $G^\rho$ the
mollified Green function on $B_R$. Then by inserting
$\varphi=G^\rho(\cdot,y)$ into (\ref{har3}), where $y$ is an
arbitrary point in $B_{\f{R}{4}}$, we have
\begin{equation}\label{har4}\aligned
\int_{B_R}\n G^\rho(\cdot,y)\cdot\n(\eta h)*1=&-\int_{B_R}G^\rho(\cdot,y)\n\eta\cdot \n h*1\\
                                              &-\int_{B_R}\eta G^\rho(\cdot,y)\De h*1\\
                                              &+\int_{B_R}h \n G^\rho(\cdot,y)\cdot\n\eta*1.
\endaligned
\end{equation}
We write (\ref{har4}) as
\begin{equation}
I_\rho=II_\rho+III_\rho+IV_\rho.
\end{equation}

By (\ref{gr}), we arrive at
\begin{equation}\label{es10}
I_\rho=\aint{B_\rho(y)}\eta h=\aint{B_\rho(y)}h.
\end{equation}
If we choose for $h$ its Lebesgue representative, then we can find a
subsequence of the $\rho$s with the property that
\begin{equation}\label{es12}
\lim_{\rho\ra 0}I_\rho=h(y).
\end{equation}

Put $T_R:=B_R-\bar{B}_{\f{R}{2}}$. Since $\n\eta\equiv 0$ outside
$T_R$,
\begin{equation}\label{es20}\aligned
|II_\rho|=&\Big|\int_{T_R}G^\rho(\cdot,y)\n\eta\cdot \n h*1\Big|
\leq\int_{T_R}G^\rho(\cdot,y)|\n\eta||\n h|*1\\
\leq&c_0R^{-1}\sup_K|\n^N H|\int_{T_R}G^\rho(\cdot,y)|d u|*1
\endaligned
\end{equation}
By (\ref{gr1}), $G^\rho(\cdot,y)\leq C_1\f{R^2}{V(R)}$ on $T_R$,
hence
\begin{equation}\label{es2}\aligned
|II_\rho|&\leq C\sup_K|\n^N H|\f{R}{V(R)}\int_{T_R}|d u|*1\\
&\leq C\sup_K|\n^N H|\f{R}{V(R)}\Big(\int_{T_R}|d u|^2*1\Big)^{\f{1}{2}}\V(T_R)^{\f{1}{2}}\\
&\leq c_1(K_1,K_2)\sup_K |\n^N H|\Big(\f{R^2}{V(R)}\int_{B_R}|d u|^2*1\Big)^{\f{1}{2}}
\endaligned
\end{equation}

Obviously
\begin{equation}\label{es3}
III_\rho=-\int_{T_R}\eta G^\rho(\cdot,y)\De h*1-\int_{B_{\f{R}{2}}}G^\rho(\cdot,y)\De h*1.
\end{equation}
According to (\ref{gr1}) and  (\ref{tele}) in Lemma \ref{l5},
\begin{equation}\label{es31}\aligned
\Big|-\int_{T_R}\eta G^\rho(\cdot,y)\De h*1\Big|&=\Big|\int_{T_R}\eta G^\rho(\cdot,y)\Hess\ H(u_*e_\a,u_*e_\a)*1\Big|\\
&\leq \sup_K|\Hess\ H|\int_{T_R}G^\rho(\cdot,y)|d u|^2*1\\
&\le C_1\sup_K|\Hess\ H|\f{R^2}{V(R)}\int_{B_R}|d u|^2*1\\
&\leq c_2\sup_K|\Hess\ H|\Big(\f{R^2}{V(R)}\int_{B_R}|d
u|^2*1\Big)^{\f{1}{2}}.
\endaligned
\end{equation}
Here $c_2$ is a positive constant depending on
$K_0,K_1,K_2$ and $\sup_K F-\inf_K F$.

Additionally assume that $K\subset N$ is diffeomorphic to a convex
domain $V$ of $\R^n$, $\chi$ denotes the diffeomorphism from $V$ to
$K$; and suppose that there exist positive constants $K_3$ and $K_4$, such
that for arbitrary $X\in TV$,
\begin{equation}\label{bound}
K_3|X|\leq |\chi_*(X)|\leq K_4|X|.
\end{equation}
Then  $u\in L^1(B_R, K)$ can be viewed as an $L^1$-function
from $B_R$ into $V\subset \R^n$. Define the mean value of $u$ on
$B_R$ by
\begin{equation}\label{mean}
\bar{u}_R:=\chi\Big[\f{\int_{B_R}(\chi^{-1}\circ
u)*1}{V(R)}\Big].
\end{equation}
Applying the Neumann-Poincar\'{e} inequality yields
\begin{equation}\label{poin}\aligned
\int_{B_R}d_N^2(u,\bar{u}_R)*1&\leq K_4^2\int_{B_R}|\chi^{-1}(u)-\chi^{-1}(\bar{u}_R)|^2*1\\
&\leq K_4^2K_2R^2\int_{B_{R}}|d(\chi^{-1}\circ u)|^2*1\\
&\leq \f{K_4^2K_2}{K_3^2}R^2\int_{B_{R}}|d u|^2*1.
\endaligned
\end{equation}
Here $d_N$ denotes the distance
function on $N$ induced by the metric. Now we write
$$h=H(u)=H(\bar{u}_R)+\big(H(u)-H(\bar{u}_R)\big),$$
 then
\begin{equation}\label{es4}
IV_\rho=H(\bar{u}_R)\int_{B_R}\n G^\rho(\cdot,y)\cdot \n
\eta*1+\int_{T_R}\big(H(u)-H(\bar{u}_R)\big)\n G^\rho(\cdot,y)\cdot
\n\eta*1.
\end{equation}
Similar to (\ref{es10})-(\ref{es12}), the first term can be
estimated by
\begin{equation}\label{es41}
\lim_{\rho\ra 0}H(\bar{u}_R)\int_{B_R}\n G^\rho(\cdot,y)\cdot \n
\eta*1= H(\bar{u}_R).
\end{equation}
We recall $|\n \eta|\leq \f{c_0}{R}$. In conjunction with
(\ref{poin}) we have
\begin{equation}\label{es42}\aligned
&\Big|\int_{T_R}\big(H(u)-H(\bar{u}_R)\big)\n G^\rho(\cdot,y)\cdot \n\eta*1\Big|\\
\leq& c_0\sup_K|\n^N H|R^{-1}\int_{T_R}d_N(u,\bar{u}_R)\big|\n G^\rho(\cdot,y)\big|*1\\
\leq& c_0\sup_K|\n^N H|R^{-1}\Big(\int_{T_R}d_N^2(u,\bar{u}_R)*1\Big)^{\f{1}{2}}\Big(\int_{T_R}\big|\n G^\rho(\cdot,y)\big|^2\Big)^{\f{1}{2}}\\
\leq& c_3(K_2,K_3,K_4)\sup_K|\n^N
H|\Big(\int_{B_{R}}|du|^2*1\Big)^{\f{1}{2}}\Big(\int_{T_R}\big|\n
G^\rho(\cdot,y)\big|^2\Big)^{\f{1}{2}}.
\endaligned
\end{equation}
Substituting (\ref{gr2}) into (\ref{es42}) implies
\begin{equation}\label{es44}
\Big|\int_{T_R}\big(H(u)-H(\bar{u}_R)\big)\n G^\rho(\cdot,y)\cdot
\n\eta*1\Big|\leq c_4\sup_K|\n^N
H|\Big(\f{R^2}{V(R)}\int_{B_{R}}|du|^2*1\Big)^{\f{1}{2}}
\end{equation}
where $c_4$ is a positive constant depending on $K_1,K_2,K_3$ and
$K_4.$

From (\ref{har4}), (\ref{es10})-(\ref{es12}), (\ref{es2}),
(\ref{es3})-(\ref{es31}), (\ref{es4})-(\ref{es41}), (\ref{es44}),
letting $\rho\ra 0$ we arrive at the following important formula
\begin{equation}\label{es}\aligned
h(y)=H(u(y))\leq& H(\bar{u}_R)+C_{7}(\sup_K |\n^N H|+\sup_K |\Hess\ H|)\Big(\f{R^2}{V(R)}\int_{B_{R}}|d u|^2*1\Big)^{\f{1}{2}}\\
&-\liminf_{\rho\ra
0}\int_{B_{\f{R}{2}}}G^\rho(\cdot,y)\De h*1
\endaligned
\end{equation}
for arbitrary $y\in B_{\f{R}{4}}$. Here
$C_7$ is a positive constant depending on $
K_0,K_1,K_2,K_3,K_4$ and $\sup_K F-\inf_K F$.

\medskip
\subsection{Image shrinking property}

Based on the convex functions constructed in  Theorem \ref{p0},
 with the aid of (\ref{es}) and Lemma \ref{l5}, we can derive
an image shrinking property for weakly harmonic maps, that is, when we
make the domain smaller, the image also gets smaller in a controlled
manner.

Recall that on $\Bbb{V}=S^n\backslash \ol{S}^{n-1}_+$ there are
well defined functions $v$ and $\varphi$ (see (\ref{angle})).

\begin{thm}\label{t3}

Let $M$ be a Riemannian manifold satisfying the DVP-condition, $K$ be
an arbitrary compact subset of $\Bbb{V}=S^n\backslash
\ol{S}^{n-1}_+$, and we put
\begin{equation}\label{c}
c:=\min\Big\{\f{1}{3} \inf_{x\in K} v, \inf_{x\in K} \varphi,
\inf_{x\in K} (2\pi-\varphi)\Big\}.
\end{equation}
If $u\in H_{loc}^{1,2}(M,S^n)$ is a weakly harmonic map and there exist
$y_0\in M$ and $R_1\leq \f{1}{3}R_0$ with
$B_{3R_1}\subset\subset M$ and $u(B_{R_1}(y_0))\subset K$, then
there exists $\de_1\in (0,1)$, only depending on $K_1,K_2$ and
$c$, such that $u(B_{\de_1 R_1}(y_0))$ is contained in a geodesic
ball in $S^n$ of radius $\arccos(\f{3}{2}c)<\f{\pi}{2}$.
\end{thm}

\begin{proof}

The definition of $\chi:(0,2\pi)\times \Bbb{D}^{n-1}\ra \Bbb{V}$
is shown in Observation 2.2. From (\ref{c}), it is easily seen
that
\begin{equation}
\chi^{-1}(K)\subset [c,2\pi-c]\times \Bbb{D}^{n-1}(\sqrt{1-(3c)^2}).
\end{equation}
Denote
\begin{equation}
\td{K}:=\chi\big([c,2\pi-c]\times
\Bbb{D}^{n-1}(\sqrt{1-(3c)^2})\big),
\end{equation}
then $\td{K}\supset K$ is diffeomorphic to a compact and convex subset of
$\R^n$. Then we can define the mean value of $u$ on $B_R$ (denoted
by $\bar{u}_R$) as in (\ref{mean}) ; and obviously $\bar{u}_R\in
\td{K}$. Also note that the constants $K_3$ and $K_4$ given in (\ref{bound})  depend
only on $c$.

Let $F$ be the convex function on $\td{K}$ given in Theorem
\ref{p0}, then there exists $K_0>0$, such that $\Hess\ F\geq K_0\
h$ on $\td{K}$, where $h$ is the canonical metric on $S^n$.
$K_0$ and $\sup_{\td{K}} F-\inf_{\td{K}} F$
 depend only on $c$.

If $U$ is defined as in (\ref{U}), then $\td{K}\subset
U$. Let $F(\cdot,\varphi_0)$ ($\varphi_0\in [c,2\pi-c]$) be
the smooth family of the smooth functions on $U$ constructed in
Theorem \ref{p0}. Put
$$\Si:=\big\{(x,\varphi_0)\in \td{K}\times [c,2\pi-c]: (x_1,x_2)=\sqrt{x_1^2+x_2^2}(\sin\varphi_0,\cos\varphi_0)\big\},$$
then for
each $(x,\varphi_0)\in \Si$, for
$x_{\varphi_0}:=(\sin\varphi_0,\cos\varphi_0,0,\cdots,0)$, then
$(x,x_{\varphi_0})=\sqrt{x_1^2+x_2^2}\geq 3c$; which implies
$F|_{\Si}<1$. Hence by the compactness of $\Si$, we can find a
positive constant $c_5$, such that
\begin{equation}
F|_\Si\leq 1-c_5.
\end{equation}
Put
\begin{equation}
c_6:=\sup_{\varphi_0\in [c,2\pi-c]}\Big(\sup_{\td{K}} \big|\n
F(\cdot,\varphi_0)\big|+\sup_{\td{K}} \big|\Hess\ F(\cdot,
\varphi_0)\big|\Big).
\end{equation}
Then for $\ep=c_5^2c_6^{-2}C_{7}^{-2}$, Lemma \ref{l5} enables us
to find $R\in [\f{1}{2}\exp(-C_6\ep^{-1})R_1,\f{1}{2}R_1]$, where
$C_6$ is a positive constant only depending on $K_1,K_2$ and $c$
, such that
\begin{equation}
\f{(2R)^2}{V(R)}\int_{B_R}|d u|^2*1\leq \ep.
\end{equation}
Since $\bar{u}_R\in \td{K}$, there is
$\varphi_0\in [c,2\pi-c]$ satisfying $(\bar{u}_R,\varphi_0)\in \Si$,
and moreover (\ref{es}) yields
\begin{equation}\aligned
F(u(y),\varphi_0)&\leq F(\bar{u}_R,\varphi_0)+\f{1}{2}C_{7}c_6\ep^{\f{1}{2}}\\
&\qquad -\liminf_{\rho\ra 0}\int_{B_{\f{R}{2}}}G^\rho(\cdot,y)\De \big(F(\cdot,\varphi_0)\circ u\big)*1\\
                &\leq 1-c_5+\f{1}{2}C_{7}c_6\ep^{\f{1}{2}}=1-\f{1}{2}c_5<1
                \endaligned
\end{equation}
for all $y\in B_{\f{R}{4}}$. Hence if we take
$\de_1=\f{1}{8}\exp(-C_6\ep^{-1})$, then for arbitrary $y\in
B_{\de_1R_1}\subset B_{\f{R}{4}}$, we have $(u(y),x_{\varphi_0})>\f{3}{2}c$;
i.e., $u(B_{\de_1 R_1})$ is contained in the geodesic ball centered
at $x_{\varphi_0}$ and of radius $\arccos(\f{3}{2}c)$.

\end{proof}

\medskip
\subsection{Estimating the oscillation}

Now we put $R^{(0)}:=\de_1 R_1$, $r_0:=\arccos(\f{3}{2}c)$,
$x^{(0)}:=x_{\varphi_0}$ and denote by $\mc{B}_{r}(x)$ the geodesic ball of
$S^n$ centered at $x$ and of radius $r$. First of all, note that
one can find positive constants $c_7, c_8$ and $c_9$, only
depending on $c$, such that
\begin{equation}
\sup_{\td{K}} |\n \rho(x,\cdot)^2|+\sup_{\td{K}} |\Hess\ \rho(x,\cdot)^2|\leq
c_7
\end{equation}
and
\begin{equation}
c_8\leq |d\exp_{x}|\leq c_9\qquad \mbox{on }\bar{\Bbb{D}}(r_0)
\end{equation}
for all $x\in S^n$. Here $\rho(x,\cdot)$ is the distance function on
$S^n$ from $x$, and $\exp_{x}$ denotes the exponential mapping of
$S^n$ at $x$; its restriction on $\bar{\Bbb{D}}(r_0)$ is a
diffeomorphism. Since
$$u(B_{R^{(0)}})\subset \mc{B}_{r_0}(x^{(0)})=\exp_{x^{(0)}}(\Bbb{D}(r_0)),$$
we can define the mean value of $u$ on $B_R$ with $R\leq R^{(0)}$ by
$$\bar{u}_R=\exp_{x^{(0)}}\Big[\f{\int_{B_R}(\exp_{x^{(0)}}^{-1}\circ u)*1}{V(R)}\Big],$$
and
$u(B_{R^{(0)}})\subset \mc{B}_{r_0}(x^{(0)})$ implies $\bar{u}_R\in
\mc{B}_{r_0}(x^{(0)})$.  Hence by (\ref{es}), there is a positive
constant $c_{10}$, only depending on $K_1,K_2$ and $c$,
such that for all $R\leq R^{(0)}$, $y\in B_{\f{R}{4}}$ and $x\in
S^n$
\begin{equation}\label{es0}
\aligned
\rho(x,u(y))^2\leq& \rho(x,\bar{u}_R)^2+c_{10}c_{7}\Big(\f{R^2}{V(R)}\int_{B_{R}}|d u|^2*1\Big)^{\f{1}{2}}\\
&-\liminf_{\rho\ra
0}\int_{B_{\f{R}{2}}}G^\rho(\cdot,y)\De\big(
\rho(x,\cdot)^2\circ u\big)*1.
\endaligned
\end{equation}
For arbitrary $\ep>0$, Lemma \ref{l5} enables us to find $R\in
[4R^{(1)},R^{(0)}]$, where
$R^{(1)}=\de^{(0)}(\ep,K_1,K_2,c)R^{(0)}$, such that
$$c_{10}c_{7}\Big(\f{R^2}{V(R)}\int_{B_{R}}|d u|^2*1\Big)^{\f{1}{2}}\leq \ep.$$
Since $\bar{u}_R\in \mc{B}_{r_0}(x^{(0)})$, one can easily
find $x^{(1)}\in \mc{B}_{r_0}(x^{(0)})$, such that
$$\rho(x^{(0)},x^{(1)})\leq \f{\pi}{2}-r_0\qquad\mbox{and}\qquad\rho(x^{(1)},\bar{u}_R)\leq 2r_0-\f{\pi}{2}.$$
Thereby $\mc{B}_{r_0}(x^{(0)})\subset
\mc{B}_{\f{\pi}{2}}(x^{(1)})$ and hence $\rho(x^{(1)},\cdot)^2$
is convex on $u(B_{R^{(0)}})\supset u(B_{\f{R}{2}})$. Letting
$x=x^{(1)}$ in (\ref{es0}) yields
$$\rho(x^{(1)},u(y))^2\leq (2r_0-\f{\pi}{2})^2+\ep\qquad \forall y\in B_{R^{(1)}}\subset B_{\f{R}{4}}.$$
Let
$$\ep=(\f{3}{2}r_0-\f{\pi}{4})^2-(2r_0-\f{\pi}{2})^2,$$
then we arrive at
\begin{equation}
u(B_{R^{(1)}})\subset \mc{B}_{r_1}(x^{(1)})\qquad \mbox{where }
r_1=\f{3}{2}r_0-\f{\pi}{4}.
\end{equation}
Similarly for each $j\geq 1$, if $r_j>\f{\pi}{4}$, we can find
$x^{(j+1)}\in \mc{B}_{r_j}(x^{(j)})$ and
$R^{(j+1)}=\de^{(j)}(K_1,K_2,c)R^{(j)}$, such that
\begin{equation}
u(B_{R^{(j+1)}})\subset \mc{B}_{r_{j+1}}(x^{(j)})\qquad \mbox{where
} r_{j+1}=\f{3}{2}r_j-\f{\pi}{4}.
\end{equation}
Noting that $r_0>r_1>r_2>\cdots$ and
$r_j-r_{j+1}=\f{1}{2}(\f{\pi}{2}-r_j)\geq \f{1}{2}(\f{\pi}{2}-r_0)$,
after $k$ steps ($k$ only depending on $c$) we can arrive at
\begin{equation}
u(B_{R^{(k)}})\subset \mc{B}_{r_{k}}(x^{(k)})\subset
\mc{B}_{\f{\pi}{4}}(x^{(k)}).
\end{equation}
This implies that for arbitrary $R\leq R^{(k)}$, $\bar{u}_R\subset
\mc{B}_{\f{\pi}{4}}(x^{(k)})$ and moreover $u(B_{R^{(k)}})\subset
\mc{B}_{\f{\pi}{2}}(\bar{u}_R)$. Hence $\rho(\bar{u}_R,\cdot)^2$ is
convex on $u(B_{R^{(k)}})$. Letting $x=\bar{u}_R$ in (\ref{es0})
yields
\begin{equation}
\rho(\bar{u}_R,u(y))^2\leq
c_{10}c_{7}\Big(\f{R^2}{V(R)}\int_{B_{R}}|d
u|^2*1\Big)^{\f{1}{2}}.
\end{equation}
for arbitrary $y\in B_{\f{R}{4}}$. Hence the oscillation of $u$ on
$B_{\f{R}{4}}$ can be controlled by
\begin{equation}
\mbox{osc}_{B_{\f{R}{4}}}u\leq
2(c_{10}c_{7})^{\f{1}{2}}\Big(\f{R^2}{V(R)}\int_{B_{R}}|d
u|^2*1\Big)^{\f{1}{4}}.\qquad \forall R\leq R^{(k)}.
\end{equation}
Again applying Lemma \ref{l5} we have the following theorem:

\begin{thm}\label{t5}

When $M,u,\Bbb{V},K,c,y_0,R_1$
satisfy the assumptions of Theorem \ref{t3}, then $u$ is continuous at $y_0$.
More precisely, for arbitrary $\ep>0$, there is $\de_2\in (0,1)$,
only depending on $K_1,K_2,c$ and $\ep$, such that
\begin{equation}
\mbox{osc}_{B_{\de_2 R_1}(y_0)}u\leq \ep.
\end{equation}
\end{thm}

\medskip
\subsection{H\"{o}lder estimates}

Now we can proceed as in \cite{J} and  Ch. 7.6  in \cite{J1} to
get the H\"{o}lder estimates for weakly harmonic maps.

By Theorem \ref{t5}, there is a constant $\de_2\in (0,1)$, depending only on $K_1,K_2$
and $c$, such that
$$u(B_{R_2}(y_0))\subset \mc{B}_{\f{\pi}{8}}(x_0)\qquad \mbox{where }R_2=\de_2 R_1,x_0=u(y_0).$$
This implies that the function $\rho^2(\cdot,x)$ is strictly convex on $u(B_{R_2})$ for arbitrary
$x$ in the convex hull of $u(B_{R_2})$; furthermore, one can find a positive constant $c_{11}$, independent
of the choice of $x$, such that
\begin{equation}
\Hess\ \rho^2(\cdot,x)\geq c_{11}h.
\end{equation}
Similar to the above, one can define the mean value of $u$ on $B_R$ with $R\leq R_2$ by
$$\bar{u}_R=\exp_{x_0}\Big[\f{\int_{B_R}(\exp_{x_0}^{-1}\circ u)*1}{V(R)}\Big].$$
Then $\bar{u}_R$ lies in the convex hull of $u(B_{R_2})$. The convexity of $\rho^2(\cdot,x)$
implies that the compositions
$$v=\rho^2(\cdot,x_0)\circ u,\qquad\mbox{and }w=\rho^2(\cdot,\bar{u}_{\f{R}{2}})\circ u$$
are both subharmonic functions. Applying Corollary
\ref{c1} yields
\begin{equation}\label{H1}
w_{+,\ep^k R}\leq \ep^2 w_{+,R}+(1-\ep^2)\bar w_{R'}
\end{equation}
for some $R'\in [\ep^k R,\f{R}{2}]$. By (\ref{poin}), the doubling
property and the Telescoping Lemma \ref{l5}, we arrive at
\begin{equation}\label{H2}\aligned
\bar w_{R'}&=\aint{B_{R'}}\rho^2(u,\bar{u}_{\f{R}{2}})
\leq \f{1}{V(R')}\int_{B_{\f{R}{2}}}\rho^2(u,\bar{u}_{\f{R}{2}})*1\\
&\leq \f{CR^2}{V(R')}\int_{B_{\f{R}{2}}}|du|^2*1\leq \f{CR^2}{V(\f{R}{2})}\int_{B_{\f{R}{2}}}|du|^2*1\\
&\leq c_{12}(v_{+,R}-v_{+,\f{R}{2}}).
\endaligned
\end{equation}
Here $c_{12}$ depends on $\ep$. With the aid of the triangle
inequality, it is easily seen that $v_{+,\ep^k R}\leq 4w_{+,\ep^k
R}$ and $w_{+,R}\leq 4v_{+,R}$. Substituting (\ref{H2})
into (\ref{H1}) yields
\begin{equation}\aligned
v_{+,\ep^k R}&\leq 4w_{+,\ep^k R}\leq 4\ep^2w_{+,R}+4(1-\ep^2)w_{R'}\\
             &\leq 16\ep^2 v_{+,R}+4c_{12}(1-\ep^2)(v_{+,R}-v_{+,\f{R}{2}})\\
             &\leq 16\ep^2 v_{+,R}+c_{13}(v_{+,R}-v_{+,\ep^k R})
             \endaligned
\end{equation}
where $c_{13}$ is a positive constant depending on $\ep$. Take $\ep=\f{1}{8}$, and put $\de=\ep^k$, then
\begin{equation}
v_{+,\de R}\leq \f{\f{1}{4}+c_{13}}{1+c_{13}}v_{+,R}.
\end{equation}
By iteration, we arrive at
\begin{equation}\label{H3}
\sup_{y\in B_R(y_0)}\rho\big(u(y),u(y_0)\big)\leq c_{14}\Big(\f{R}{R_2}\Big)^\si=c_{14}\de_2^{-\si}R_1^{-\si}R^\si
\end{equation}
for arbitrary $0<R\leq R_2$,
where $c_{14}>0$, $\si\in (0,1)$ are constants depending only on $K_1, K_2$ and $c$.

We note that for each $y_1\in B_{\f{R_1}{2}}(y_0)$, $u(B_{\f{R_1}{2}}(y_1))\subset u(B_{R_1}(y_0))\subset K$;
therefore, (\ref{H3}) still holds true for arbitrary $0<R\leq \f{R_2}{2}$
when $y_0$ is replaced by $y_1$. We can get the following theorem:

\begin{thm}\label{t4}
With $M,u,\Bbb{V},K,c,y_0,R_1$
satisfying the above assumptions, there exist numbers $\si\in (0,1)$,
$\de_3\in (0,1)$ and $C_8>0$, depending only on
$K_1,K_2$ and $c$, such that the $\si$-H\"{o}lder seminorm
of $u$ on $\bar{B}_{\de_3 R_1}(y_0)$ can be estimated by
\begin{equation}
[u]_{C^\si(\bar{B}_{\de_3 R_1})(y_0)}\leq C_{8}R_1^{-\si}.
\end{equation}
\end{thm}

Here and in the sequel,
\begin{equation}\label{holder}
[u]_{C^\si(S)}:=\sup_{y_1,y_2\in S,y_1\neq y_2}\f{\rho\big(u(y_1),u(y_2)\big)}{r(y_1,y_2)^\si}
\end{equation}
where $r(\cdot,\cdot)$ is the distance function on $M$ induced by the
metric $g$.

We point out that the H\"older bound thus depends only on the geometry
of the domain, as incorporated in the volume doubling constant $K_1$
of (\ref{double}) and the Poincar\'e inequality constant $K_2$ of
(\ref{str-p}) and on the convexity condition on the image as reflected
in the constant $c$ that controls the Hessian of our convex functions,
but not on the map $u$, and in particular not on its
energy. Therefore, in the sequel, in our Liouville theorem, we do not
need to require that the map in question have finite energy.

\medskip
\subsection{A Liouville type theorem}

Letting $R_1\ra +\infty$ in Theorem \ref{t5} or Theorem \ref{t4}, we obtain
the following Liouville-type theorem:

\begin{thm}\label{t7}

Let $(M,g)$ be a complete Riemannian manifold satisfying the
DVP-condition with $R_0=+\infty$, and $K$ be an arbitrary compact
subset of $\Bbb{V}=S^n\backslash \ol{S}^{n-1}_+$. If $u:M\ra S^n$
is a weakly harmonic map, and almost every $y\in M$ satisfies
$u(y)\in K$, then $u$ is constant.

\end{thm}

{\bf Remark:} Solomon \cite{so2} obtained regularity and Liouville
theorems for {\bf energy
minimizing} harmonic maps when the image omits a neighborhood of a
totally geodesics $S^{n-2}$. Since we wish to apply our Liouville
theorem to the Bernstein problem in the next section, we cannot make
the assumption that the harmonic maps under consideration be energy
minimizing as in general it is not clear under which conditions Gauss
maps are energy minimizing.  In any case, his Liouville theorem is
only derived for the case where the domain is Euclidean space. And Solomon's result ceases to be true without the
energy minimizing assumption.

\bigskip\bigskip

\Section{Analytic and geometric conclusions}{Analytic and geometric conclusions}

\subsection{Simple manifolds}

Let $M=\Bbb{D}^m(r_0)\subset \R^m$ with metric
$g=g_{\a\be}(y^1,\cdots,y^m)dy^\a dy^\be$. We suppose that  there exist
two positive constants $\la$ and $\mu$, such that
\begin{equation}\label{uni}
\la^2 |\xi|^2\leq g_{\a\be}\xi^\a\xi^\be\leq \mu^2|\xi|^2.
\end{equation}
Now we define a distance function $d$ on $M$:
\begin{equation}
(y_1,y_2)\in M\times M\mapsto \la|y_1-y_2|.
\end{equation}
Then obviously
$$r(y_1,y_2)\geq \la|y_1-y_2|=d(y_1,y_2)$$
where $r(\cdot,\cdot)$ denotes the usual distance function induced by the metric $g$.

Denoting by $\Bbb{D}(y,R)$ the Euclidean disk centered at $y$ and of radius $R$,
then with respect to the distance function $d$,
\begin{equation}
B_R(y)=\Bbb{D}(y,\la^{-1} R)\cap \Bbb{D}(r_0).
\end{equation}
Taking
\begin{equation}
R_0:=\la r_0,
\end{equation}
then for arbitrary $R\leq R_0$,
\begin{equation}
V(y,2R)= \V\big(\Bbb{D}\big(y,2\la^{-1} R)\cap \Bbb{D}(r_0)\big)\leq \big(\f{2\mu}{\la}\big)^m R^m \om_m
\end{equation}
and
\begin{equation}
V(y,R)\geq \V\Big(\Bbb{D}\big(y-\f{1}{2}\la^{-1}R\f{y}{|y|},\f{1}{2}\la^{-1} R\big)\Big)\geq 2^{-m}R^m \om_m.
\end{equation}
Here $\om_m$ is the volume of the $m$-dimensional Euclidean unit
disk equipped with the canonical metric. Hence $(M,g)$ satisfies
condition (V) with $K_1=(\f{4\mu}{\la})^m$.

We note that
$B_R(y)\subset\subset M$ if and only if $B_R(y)=\Bbb{D}(y,\la^{-1}R)$. By \cite{p-w},
\begin{equation}
\int_{\Bbb{D}(y,\la^{-1}R)}|v-\bar{v}_{\la^{-1}R}|^2dy\leq 4\pi^{-2}\la^{-2}R^2\int_{\Bbb{D}(y,\la^{-1}R)}|Dv|^2 dy
\end{equation}
where
$$\bar{v}_{\la^{-1}R}=\f{\int_{\Bbb{D}(y,\la^{-1}R)} v\ dy}{\int_{\Bbb{D}(y,\la^{-1}R)}  dy}$$
and $|D v|^2=\sum_\a (\p^\a v)^2$. Noting that $*1=\sqrt{\det(g_{\a\be})}dy$ and $|\n v|^2=g^{\a\be}\p^\a v\p^\be v$,
where $\big(g^{\a\be}\big)$ denotes the inverse matrix of $\big(g_{\a\be}\big)$, it is easy for us to arrive at
\begin{equation}\aligned
\int_{B_R(y)}|v-\bar{v}_{B_R(y)}|^2*1&\leq \mu^m \int_{\Bbb{D}(y,\la^{-1}R)}|v-\bar{v}_{\la^{-1}R}|^2dy\\
                                     &\leq 4\pi^{-2}\la^{-2}\mu^m R^2\int_{\Bbb{D}(y,\la^{-1}R)}|Dv|^2 dy\\
                                     &\leq 4\pi^{-2}\big(\f{\mu}{\la}\big)^{m+2}R^2\int_{B_R(y)}|\n v|^2*1
\endaligned
\end{equation}
which means that  $M$ satisfies condition (P) with
$K_2=4\pi^{-2}\big(\f{\mu}{\la}\big)^{m+2}$.

Therefore, applying Theorem \ref{t4} yields the following H\"{o}lder estimate:

\begin{thm}\label{t6}
Let $M=\Bbb{D}^m(r_0)$ with metric $g=g_{\a\be}dy^\a dy^\be$, and
suppose that there exist
two positive constants $\la$ and $\mu$, such that
\begin{equation*}
\la^2 |\xi|^2\leq g_{\a\be}\xi^\a\xi^\be\leq \mu^2|\xi|^2
\end{equation*}
for arbitrary $\xi\in M$. Suppose $u\in H_{loc}^{1,2}(M,S^n)$ is a weakly harmonic map, and there
exists a compact set $K\subset S^n\backslash \ol{S}^{n-1}_+$, such that $u(y)\in K$ for almost every $y\in M$.
Then there exist numbers $\si_1\in (0,1)$,
$\ep_4\in (0,1)$, and $C_9>0$, depending only on $m,\f{\mu}{\la}$ and $K$, but not on $r_0$, such
that the $\si_1$-H\"{o}lder seminorm of $u$ on $\bar{B}_{\ep_4 r_0}$ is estimated by
\begin{equation}
[u]_{C^{\si_1}(\bar{B}_{\ep_4 r_0})}\leq C_9 (\la r_0)^{-\si_1}.
\end{equation}
\end{thm}

{\bf Remark:}
Here the definition of H\"{o}lder seminorm is the same as (\ref{holder}). In many references, e.g. \cite{h-j-w},
the $\si$-H\"{o}lder seminorm is given by
\begin{equation}
[u]_{C^\si(S)}:=\sup_{y_1,y_2\in S,y_1\neq y_2}\f{\rho\big(u(y_1),u(y_2)\big)}{|y_1-y_2|^\si}.
\end{equation}
With that definition, the corresponding estimate would read as
\begin{equation}
[u]_{C^{\si_1}(\bar{B}_{\ep_4 r_0})}\leq C_9  r_0^{-\si_1}.
\end{equation}

If $(M,g)$ is a Riemannian manifold, then every point $y_0\in M$ has a coordinate patch
with induced metric, hence from Theorem \ref{t6} the following estimate immediately follows:

\begin{thm}

Let $(M,g)$ be an  $m$-dimensional Riemannian manifold,
$u\in H_{loc}^{1,2}(M,S^n)$ be a weakly harmonic map, and $K$ be a compact subset of
$S^n\backslash \ol{S}^{n-1}_+$, with the property that  almost every $y\in M$ satisfies $u(y)\in K$. Then for any compact
subset $S$ of $M$, there exist numbers $\si_2\in (0,1)$,
and $C_{10}>0$, depending on $m,K,S$, and on the metric of $M$, but not on $u$, such that the estimate
\begin{equation}
[u]_{C^{\si_2}(S)}\leq C_{10}
\end{equation}
holds. Moreover, if $M$ is homogeneously regular (in the sense of Morrey) with constants $\la$ and $\mu$,
then $\si_2$ depends only on $m,\la,\mu$ and $K$, while $C_{10}$ depends, apart from these parameters, also on
$S$.

\end{thm}

We use here a slightly altered definition for \textbf{homogenously regular} manifolds (cf. \cite{mo}, p.363):
A $C^1$-manifold $M$ is said to be homogenously regular if there exist positive numbers $\la$ and $\mu$,
such that each point $y_0$ of $M$ is the center of a coordinate patch $\{z:|z|\leq 1\}$ for which
$$\la^2 |\xi|^2\leq g_{\a\be}(z)\xi^\a\xi^\be\leq \mu^2|\xi|^2$$
holds for all $\xi\in \R^m$ and each $|z|\leq 1$.

Recall that a Riemannian manifold $M$ is said to be \textbf{simple},
if it is described by a single set of coordinates $y$
on $\R^m$ and by a metric
$$g=g_{\a\be}(y)dy^\a dy^\be$$
for which there exist positive numbers $\la$ and $\mu$ such that
$$\la^2 |\xi|^2\leq g_{\a\be}(y)\xi^\a\xi^\be\leq \mu^2|\xi|^2$$
holds for all $\xi,y\in \R^m$.

Applying Theorem \ref{t7} yields the following Liouville-type theorem:

\begin{thm}

Let $u$ be a weakly harmonic map from a simple Riemannian
manifold $M$ to $S^n$. If $u(M)$ is contained in a compact subset
of $S^n\backslash \ol{S}^{n-1}_+$, then $u$ has to be a constant
map.
\end{thm}

{\bf Remark:}
Let  $M$ be an entire graph given by $f:\ir{n}\to\ir{}.$ If $|\n
f|\le\be<\infty$, then the induced metric on $M$ is simple. Furthermore, the image under its Gauss map lies in a
closed subset of an open hemisphere. This is the situation of \cite{M}.

For higher codimensional graphs with suitable bounded slope of
defining functions we also obtain simple manifolds with convex
Gauss image in Grassmannian manifolds. This is the situation of
\cite{h-j-w}.

\medskip
\subsection{Manifolds with nonnegative Ricci curvature}

Let $(M,g)$ be a Riemannian manifold with $\mbox{Ric}\ M\geq 0$. For
 the canonical distance function  $d(\cdot,\cdot)$ induced by $g$,
 Condition (D) is obviously satisfied. By the classical relative volume comparison theorem, $M$ enjoys the doubling property with
constant $K_1=2^m$. In \cite{bu}, P. Buser shows that $M$ satisfies the Neumann-Poincar\'{e} inequality
with constant $K_2=K_2(m)$. Therefore, Theorem \ref{t7} yields
the following Liouville-type theorem:

\begin{thm}

Let $u$ be a weakly harmonic map from a Riemannian manifold $M$
with nonnegative Ricci curvature to $S^n$. If $u(M)$ is contained
in a compact subset of $S^n\backslash \ol{S}^{n-1}_+$, then $u$
has to be a constant map.

\end{thm}

\medskip
\subsection{Bernstein type theorems}

Let $M^m\subset \R^{m+1}$ be a complete  hypersurface with the
induced Riemannian metric. Then we can define $d: M\times M\ra \R$
\begin{equation}
(y_1,y_2)\ra |y_1-y_2|
\end{equation}
where $|y_1-y_2|$ denotes the Euclidean distance from $y_1$ to
$y_2$. Obviously $d(y_1,y_2)\leq r(y_1,y_2)$. Therefore $M$
satisfied Condition (D) if and only if the inclusion map $i: M\ra
\R^{m+1}$ is injective, i.e. $M$ is an imbedded hypersurface.

Given $y\in M$ and $R>0$, the density is defined by
\begin{equation}
\Th(y,R)=\f{V(y,R)}{\om_m R^m}.
\end{equation}
Here $\om_m$ is the volume of $S^m\subset \R^{m+1}$ equipped with the canonical metric.
The following monotonicity  of the volume is well-known.

\begin{lem}
If $M^m$ is any complete minimal submanifold in Euclidean space, then
$\Th(y,R)$ is monotonically
nondecreasing in $r$ and $\lim_{R\ra 0}\Th(y,R)=1$.
\end{lem}

We say that a minimal hypersurface $M$ has Euclidean volume growth if there exist $y_0\in M$
and a positive constant $C$, such that
\begin{equation}
\Th(y_0,R)\leq C
\end{equation}
for arbitrary $R>0$. For each $y\in M$, if we denote $r=d(y,y_0)$, then
$$\Th(y,R)=\f{V(y,R)}{\om_m R^m}\leq \f{V(y_0,R+r)}{\om_m R^m}=\big(\f{R+r}{R}\big)^m \Th(y_0,R+r)$$
Letting $R\ra +\infty$ implies
$$\lim_{R\ra +\infty}\Th(y,R)\leq \lim_{R\ra +\infty}\Th(y_0,R)=C.$$
And moreover
\begin{equation}
\f{V(y,2R)}{V(y,R)}=2^m \f{\Th(y,2R)}{\Th(y,R)}\leq 2^m C.
\end{equation}
i.e. $M$ satisfies Condition (V) with $R_0=+\infty$.

The  Gauss map $\gamma : M \to S^m $ is defined by
$$
 \g (y) = T_y M \in S^m
$$
via the parallel translation in $ \ir{m+1} $ for all $ y \in M $.
Ruh-Vilms \cite{r-v} proved  that the mean curvature vector of
$M$ is parallel if and only if its Gauss map is a harmonic map.
This fact enables us to apply Theorem \ref{t7} and obtain a
Bernstein type theorem as follows.

\begin{thm}\label{b1}

Let $M^m\subset \R^{m+1}$ be a complete minimal embedded
hypersurface. Assume $M$ has Euclidean volume growth, and there
is a positive constant $C$, such that for arbitrary $y\in M$ and
$R>0$, the Neumann-Poincar\'{e} inequality
$$\int_{B_{ R}( y)}|v-\bar{v}_{B_{ R}(y)}|^2*1\leq
CR^2\int_{B_R(y)}|\n v|^2 *1$$
holds for all $v\in C^\infty(B_R(y))$, where
$$B_R(y)=\{z\in M: |z-y|<R\}.$$
If the image under the Gauss map omits a neighborhood of $\ol{S}^{n-1}_+$,  then
$M$ has to be an affine linear space.

\end{thm}

This naturally raises the question under which conditions a complete
embedded minimal hypersurfaces in Euclidean space satisfies the
Neumann-Poincar\'e inequality. So far, only partial results in this direction
seem to be known. When $M$ is an {\bf area-minimizing} hypersurface,
the Neumann-Poincar\'{e} inequality has been proved by Bombieri-Giusti (see \cite{b-g}). Hence we have:

\begin{thm}\label{b2}
Let $M^m\subset \R^{m+1}$ be a complete embedded area-minimizing
hypersurface. Assume $M$ has Euclidean volume growth. If the
image under the Gauss map omits a neighborhood of
$\ol{S}^{n-1}_+$, then $M$ has to be an affine linear space.
\end{thm}
\noindent
{\bf Remarks:}
\begin{itemize}
\item
Solomon \cite{so1} proved such a result under a somewhat weaker
condition on the Gauss image, but needed the additional topological
assumption that the first Betti number of $M$ vanishes. Presumably,
Solomon's result ceases to be true without that topological
assumption.
\item
Recently, N.Wickramasekera\cite{wick} proved new Poincar\'e type inequalities for
stable minimal hypersurfaces of dimension at most 6 in Euclidean space of controlled volume
growth. The dimensional restriction in his results is related to the
fact that in higher dimensions, stable minimal hypersurfaces may have
singularities. 

\end{itemize}

\bigskip\bigskip

\bibliographystyle{amsplain}

\end{document}